\documentclass[11pt,fleqn]{amsart}

\usepackage{graphicx}
\usepackage{latexsym,amsmath,amsthm,amssymb,amsfonts,MnSymbol,epsfig}
\usepackage[initials]{amsrefs}
\usepackage[all]{xy}

\providecommand{\CC}{{\mathbb{C}}}
\providecommand{\RR}{{\mathbb{R}}}
\providecommand{\QQ}{{\mathbb{Q}}}

\providecommand{\ZZ}{{\mathbb{Z}}}

\providecommand{\EE}{{\mathcal E}}
\providecommand{\HH}{{\mathcal H}}
\providecommand{\KK}{{\mathcal K}}
\providecommand{\LL}{{\mathcal L}}

\providecommand{\UU}{{\mathcal U}}

\newcommand{\ang}[1]{\langle #1 \rangle} 

\providecommand{\THX}{{\mathbb{T}_H X}}
\providecommand{\TX}{{\mathbb{T}X}}
\providecommand{\Lg}{{\mathfrak g}}

\newtheorem{theorem}{Theorem}[subsection]
\newtheorem{lemma}[theorem]{Lemma}
\newtheorem{corollary}[theorem]{Corollary}
\newtheorem{proposition}[theorem]{Proposition}

\theoremstyle{definition}
\newtheorem{definition}[theorem]{Definition}

\theoremstyle{remark} 
\newtheorem{remark}[theorem]{Remark}
\newtheorem{example}[theorem]{Example}

\numberwithin{equation}{section}

\begin{document}

\title{$K$-homology and index theory on contact manifolds}
\author{Paul F.\ Baum}
\address{The Pennsylvania State University, University Park, PA, 16802, USA}
\email{baum@math.psu.edu}
\author{Erik van Erp}
\address{Dartmouth College, 6188, Kemeny Hall, Hanover, New Hampshire, 03755, USA}
\curraddr{University of Hawaii at Manoa, 2565 McCarthy Mall, Honolulu, Hawaii 96822, USA}
\email{jhamvanerp@gmail.com}

\thanks{Paul Baum thanks Dartmouth College for the generous hospitality provided to him via the Edward Shapiro fund.
Erik van Erp thanks Penn State University for a number of productive and enjoyable visits.}
\thanks{PFB was partially supported by NSF grant DMS-0701184}
\thanks{EvE was partially supported by NSF grant DMS-1100570}


\maketitle

\tableofcontents

\section{Introduction}

This paper applies $K$-homology to solve the index problem for a class of hypoelliptic (but not elliptic) operators on contact manifolds.
$K$-homology is the dual theory to $K$-theory.
We explicitly calculate the $K$-cycle (i.e., the element in geometric $K$-homology) determined by any hypoelliptic Fredholm operator in the Heisenberg calculus.

The index theorem of this paper precisely indicates how the analytical versus geometrical $K$-homology setting provides an effective framework for extending formulas of Atiyah-Singer type to non-elliptic Fredholm operators.
Given an index problem, the $K$-homology framework provides a guide and hint as to what the solution of that index problem might be.

Let $P$ a differential operator on a closed contact manifold $X$ of dimension $2n+1$,
\[ P\;\colon\; C^\infty(X,F^0)\to C^\infty(X, F^1)\]
$P$ is {\em Heisenberg-elliptic} if its principal symbol in the Heisenberg calculus is invertible.
Any Heisenberg-elliptic $P$ is Fredholm and hypoelliptic, but not elliptic.
Invertibility of the principal Heisenberg symbol does not imply invertibility of the usual principal symbol used by Atiyah and Singer.
However, the analytic properties of a Heisenberg-elliptic $P$ are such that $P$ determines an element in Kasparov (i.e., analytic) $K$-homology
\[ [P]\in KK^0(C(X),\CC)\]
The main theorem of this paper is a topological formula for this analytical $K$-homology element $[P]$.
Our formula is
\vskip 6pt
\noindent {\bf Theorem.}
\begin{equation}\label{MainFormula}
[P]\;= \; [\sigma^+_H(P)]\;\cap\; [X^+]+\;[\sigma^-_H(P)]\;\cap\; [X^-]
\end{equation}
\vskip 6pt
Here  $[X^+]$ and $[X^-]$ are the fundamental cycles in the $K$-homology group $K_1(X)$ given by the two natural Spin$^c$ structures of $X$, one for each co-orientation of the contact structure.
In \eqref{MainFormula}, $\sigma^+_H(P)$ and $\sigma^-_H(P)$ are elements in the $K$-theory group $K^1(X)$ extracted from the {\em principal Heisenberg symbol} of $P$. 
More specifically, $\sigma^+_H(P)$ and $\sigma^-_H(P)$ are explicit automorphisms of two $\CC$ vector bundles on $X$ intrinsically associated to the contact structure of $X$.
The cap product $\cap$ in \eqref{MainFormula} is the $K^\bullet(X)$-module structure of $K_\bullet(X)$.

Applying the Chern character to \eqref{MainFormula} gives a characteristic class formula in rational homology $H_\bullet(X,\QQ)$ for the index of $P$.
Our proof of \eqref{MainFormula} is sufficiently canonical so that the same formula is valid for the equivariant index problem and the families index problem for the Heisenberg calculus. 

The index problem for Heisenberg-elliptic operators on contact manifolds has been studied, and special cases were solved by 
C.\ Epstein and R.\ Melrose \citelist{\cite{Me97} \cite{Ep04} \cite{Epxx}}, and E.\ van Erp  \citelist{\cite{vE10a} \cite{vE10b}}.
See also the related work of A.\ Connes and H.\ Moscovici for foliated manifolds \citelist{\cite{CM95} \cite{CM98}}.
Our work builds on these partial results. From our point of view equation \eqref{MainFormula} solves the index problem (including the equivariant and families problems) for Heisenberg-elliptic operators on contact manifolds.

\subsection*{Acknowledgment}
We thank A.\ Connes, R.\ G.\ Douglas, C.\ Epstein, E.\ Getzler, A.\ Gorokhovsky, J.\ Kohn, N.\ Higson, R.\ Melrose, R.\ Nest and M.\ Taylor for enlightening comments and discussions.

\section{The problem and the solution}

\subsection{$K$-homology}
The point of view of \cite{BD80} is that index theory, in general, is based on the equivalence between geometric $K$-homology and analytic $K$-homology.
For a topological space $X$, a geometric $K$-cycle is a triple $(M,E,\varphi)$
consisting of a closed Spin$^c$ manifold $M$,
a complex vector bundle $E$ on $M$,
and a continuous map $\varphi\,\colon M\to X$.
The collection of $(M, E, \varphi)$-cycles, subject to a certain equivalence relation, forms an abelian group under disjoint union.
We denote this group by $K^{top}_\ast(X)$.

There is a natural map
\[ \mu\;\colon\; K_j^{top}(X) \to KK^j(C(X),\CC),\qquad j=0,1\]
from geometric $K$-homology to analytic $K$-homology, defined as follows.
Let $D_M\otimes I_E$ denote the Dirac operator $D_M$ for the Spin$^c$ manifold $M$
twisted by the vector bundle $E$.
Then $\mu(M,E,\varphi)$ is the push-forward by $\varphi$ of the analytic $K$-cycle $[D_M\otimes I_E]\in KK^j(C(M),\CC)$,
\[ \mu(M, E, \varphi) = \varphi_*(D_M\otimes I_E)\in KK^j(C(X),\CC).\]
For a finite $CW$-complex $X$ the map $\mu$ is an isomorphism \cite{BHS07}.

In \cite{BD80} Baum and Douglas ask
\begin{quote} {\it Let $X$ be a finite $CW$-complex. 
Given an element in analytic $K$-homology, \\$\xi \in KK^j(C(X),\CC)$, explicitly compute the unique $\widetilde{\xi}$ in geometric $K$-homology, $\widetilde{\xi}\in K^{top}_j(X)$, corresponding to $\xi$.}
\end{quote}
As explained in \cite{BD80}, several well-known index theorems (including, of course, the Atiyah-Singer index formula for elliptic operators) can be understood in this framework.
Moreover, if the construction of the $(M,E,\varphi)$-cycle for a given class of operators is sufficiently canonical,
it will also solve the equivariant and families index problems for this class of operators.

In this paper we solve the index problem for the Heisenberg calculus in this context.
In other words, we solve the following problem.
\begin{quote}
{\it Let $X$ be a closed contact manifold.
A Heisenberg-elliptic (pseudo)differential operator $P$ on $X$ determines an element $[P]$ in analytic $K$-homology $[P]\in KK^0(C(X),\CC)$.
Explicitly compute an $(M,E,\varphi)$-cycle in geometric $K$-homology $K_0^{top}(X)$
with $\mu(M,E,\varphi)=[P]$.}
\end{quote}
Starting with $[P]\in KK^0(C(X),\CC)$ the $K$-cycle $(M,E,\varphi)$ which solves the index problem for $P$ should be constructed so that {\it mutatis mutandis} the same $K$-cycle will solve the equivariant index problem (when a compact Lie group is acting) and the index problem for families.
As an immediate corollary of the computation of the $K$-cycle one obtains characteristic class formulas for the homology Chern character of $P$ (including for equivariant operators and families of operators). 
Thus, the computation of the $K$-cycle gives a fully satisfactory answer to the index problem for the Heisenberg calculus on contact manifolds.

\subsection{Contact manifolds}

\begin{definition}
 A {\em contact structure} on a manifold $X$ of dimension $2n+1$ is a $C^\infty$ sub-vectorbundle $H\subset TX$ of fiber dimension $2n$ such that if $\theta$ is any 1-form on $X$ with $H=\mathrm{Ker}\, \theta$ then the 2-form $d\theta$ restricted to $H$ is symplectic in each fiber $H_x$. 
  \end{definition}
In all that follows $X$ will denote a closed smooth contact manifold of dimension $2n+1$.
$H$ is referred to as the contact hyperplane bundle, while $\theta$ (which has to be chosen) is a contact form.
$\theta$ is a contact form on $X$ if and only if  $\theta(d\theta)^n$ is a nowhere vanishing $2n+1$-form (i.e., a volume form) on $X$.

Once a contact form $\theta$ has been fixed, a number of geometric structures on $X$ are then determined.
The {\em Reeb vector field} $T$ is the unique vector field on $X$ with the properties
\[ \theta(T)=1,\; d\theta(T,\,\cdot\,) = 0.\]
The Reeb field $T$ is tranversal to the contact hyperplane bundle $H$.

We choose a complex structure $J$ for $H$
that is compatible with the symplectic form $d\theta$ in each fiber, i.e.,
\[ J^2=-1,\;\; d\theta(Jv,Jw)=d\theta(v,w),\;\; d\theta(Jv,v)\ge 0.\]
Such a choice of $J$ is always possible, because the space of compatible complex structures on a symplectic vector space is contractible.
The choice of $J$ also defines a Euclidean structure in the fibers of $H$, with inner product defined by $\ang{v,w}=d\theta(Jv,w)$.
In addition, $J$ and $\theta$ determine a hermitian inner product,
\[ \ang{v,w} = d\theta (Jv,w) + i\, d\theta(v,w).\]  

\subsection{Example: A hypoelliptic (but not elliptic) Fredholm operator}

A {\em sublaplacian} for the given structure $(X, \theta, J)$
is a second order differential operator on $X$ constructed as follows.
Locally, a sublaplacian is of the form
\[ \Delta_H = \sum_{j=1}^{2n} -W_j^2.\]
Here $W_1, \ldots, W_{2n}$ is a collection of vector fields, defined in an open subset $U\subseteq X$, 
such that at each point $x\in U$
the vectors $W_j(x)$ form an orthonormal frame for the Euclidean vector space $H_x$.
A global sublaplacian on $X$ is then constructed via a partition of unity. 

To obtain an operator with interesting index theory we add a lower order term.
Let $r$ be a positive integer and let $\gamma$ be a $C^\infty$ map from $X$ to the $\CC$ vector space of all $r\times r$ matrices, denoted $M(r,\CC)$,
\[ \gamma\;\colon\; X\to M(r, \CC)\]
Examples of the kind of Fredholm operator to which our result applies are the second order differential operators $P_\gamma$ of the form
\[ P_\gamma = I_r\otimes \Delta_H + i\gamma(I_r\otimes T), \qquad i=\sqrt{-1}\]
\[ P_\gamma \;\colon\; C^\infty(X,\CC^r)\to C^\infty(X,\CC^r)\]
where (as usual) $I_r$ is the $r\times r$ identity matrix.
Operators of this type have been studied extensively.
See for example \cite[Ch. 1]{BG88} or \cite[Thm. 8.1]{Ep04} for the following result.
\begin{proposition}\label{gamma}
The operator $P_\gamma$ is invertible modulo smoothing operators if for all $x\in X$
\[\gamma(x) -\lambda I_r \;\;\text{\rm is invertible for all} \;\lambda \in \{\cdots\,, -n-4,-n-2,-n,n,n+2,n+4,\cdots\,\}\]
\end{proposition}
The operators $P_\gamma$ are not elliptic, but they are hypoelliptic and Fredholm.
The parametrix (i.e., the inverse modulo smoothing operators) of these operators $P_\gamma$ is not a classical pseudodifferential operator, but it is a pseudodifferential operator in the Heisenberg calculus.
The invertibility requirement in Proposition \ref{gamma} is equivalent to invertibility of the principal Heisenberg symbol of $P_\gamma$.
We shall refer to hypoelliptic operators with a parametrix in the Heisenberg calculus as {\em Heisenberg-elliptic}.
Like $P_\gamma$, any Heisenberg-elliptic differential operator consists of a highest order part that only involves the $H$-directions, and is ``elliptic'' in the $H$-directions, plus terms of lower order.
For further details about Heisenberg-elliptic operators, see section \ref{AnalyticTriangle} below.

\subsection{The problem}
In general, we consider a Heisenberg-elliptic operator $P$, 
\[ P\;\colon\; C^\infty(X,F^0)\to C^\infty(X,F^1),\]
where $F^0, F^1$ are two smooth $\CC$ vector bundles on $X$.
Although not elliptic in the usual sense, $P$ has the same basic analytic properties as an elliptic operator.
For example, if $P$ is Heisenberg-elliptic then the symmetric operator
\[ D = \left( \begin{array}{cc} 0& P^*\\ P &0\end{array}\right)\;\colon\; C^\infty(X, F^0\oplus F^1)\to C^\infty(X, F^0\oplus F^1)\]
is essentially self-adjoint, has discrete spectrum, and finite dimensional eigenspaces consisting of smooth functions \cite[section 2.6]{vE10a}.
The bounded operator $F=D(1+D^2)^{-1/2}$ is Fredholm, and
commutators $[F, M_f]$ of $F$ with functions $f\in C(X)$ are compact. 
Therefore $F$ determines an element in the analytic $K$-homology of $X$,
\[ [P] \in KK^0(C(X),\CC).\]
We briefly verify this.
We use the standard convention that $\KK$ denotes the separable $C^*$-algebra of compact operators.

\begin{lemma}
The operator $F=D(1+D^2)^{-1/2}$ is a bounded Fredholm operator on $L^2(X)^{\oplus 2r}$ that 
satisfies the pseudolocality condition
$[F, M_f]\in \KK(L^2(X)^{\oplus 2r})$ for multiplication operators $M_f\in \LL(L^2(X)^{\oplus 2r}), f\in C(X)$. 
\end{lemma}
\begin{proof}
If $P$ is Heisenberg-elliptic then $F$ is Heisenberg-elliptic and has order zero in the Heisenberg pseudodifferential calculus.
Order zero operators in the Heisenberg calculus are bounded on $L^2(X)^{\oplus 2r}$ while operators of negative order are compact
(by the Rellich lemma for the Heisenberg calculus).
Therefore $F$ is bounded and invertible modulo compact operators.

We now verify the pseudolocality condition. 
Let $\Psi_H$ denote the  norm completion of the algebra of order zero Heisenberg operators.
The principal Heisenberg symbol $\sigma_H$  gives rise to a short exact sequence of $C^*$-algebras,
\[ 0\to\KK\to \Psi_H\stackrel{\sigma_H}{\to} S_H\to 0.\]
Without going into the details of the Heisenberg calculus, it suffices to know that the symbol algebra $S_H$ contains $C(X)$ as a central subalgebra.
Moreover, for a continuous function $f\in C(X)$ the multiplication operator $M_f$ on $L^2(X)$ is contained in $\Psi_H$, while its principal symbol is $\sigma_H(M_f)=f$, as usual.
It follows that $[\sigma_H(F), \sigma_H(M_f)]=0$ for any order zero operator $F$,
which is equivalent to $[F,M_f]\in \KK$.

\end{proof}

Our aim is to solve the index problem for Heisenberg-elliptic operators
within the $K$-homology framework formulated by Paul Baum and Ron Douglas \citelist{\cite{BD80} \cite{BD82}}, i.e., given a Heisenberg-elliptic operator $P$
to construct a $K$-cycle $(M,E,\varphi)$ with
\[ [P] = \mu(M,E,\varphi).\]

\subsection{The solution}
We now describe the solution of the index problem in $K$-homology for a Heisenberg-elliptic differential (or even pseudodifferential) operator $P$.
 
{\vskip 6pt \noindent \bf The equatorial symbol.}
Consider the {\em classical}  principal symbol of $P$,
\[ \sigma(P)\;\colon\;\pi^*F^0\to \pi^*F^1\]
The  symbol $\sigma(P)$ is a vector bundle homomorphism defined on the cosphere bundle $S(T^*X)$.
For an elliptic operator the classical principal symbol is (by definition) a vector bundle isomorphism.
For a Heisenberg-elliptic operator $\sigma(P)$ is not invertible on all of $S(T^*X)$.
\footnote{The only differential operators that are both elliptic and Heisenberg-elliptic are vector bundle isomorphisms.}

The restriction $\sigma(P)|S(H^*)$ of the principal symbol of $P$ to the equator $S(H^*)\subset S(T^*X)$ is referred to as the {\em equatorial symbol} of $P$.
Since a Heisenberg-elliptic differential operator $P$ is ``elliptic'' in the $H$ directions, 
its equatorial symbol $\sigma(P)|S(H^*)$ is invertible.
The topological information contained in the equatorial symbol is essentially trivial, and is certainly insufficient for calculating the index of $P$.
In fact, the equatorial symbol $\sigma(P)|S(H^*)$ of a Heisenberg-elliptic operator is homotopic to the pull-back to $S(H^*)$ of an isomorphism of the vector bundles $F^0, F^1$ on $X$,
\[\sigma_0\;\colon\; F^0\cong F^1\]
For example, the classical principal symbol of the operators $P_\gamma=I_r\otimes \Delta_H + i\gamma(I_r\otimes T)$ only depends on the sublaplacian $\Delta_H$, and the equatorial symbol of $P_\gamma$ is the constant map from $S(H^*)$ to the identity matrix $I_r$.

In order to state our result we fix an isomorphism $\sigma_0\colon F^0\cong F^1$ of vector bundles on $X$ determined by the equatorial symbol.
Having fixed $\sigma_0$,  the construction of the geometric $K$-cycle $[(M, E,\varphi)]=\mu^{-1}(P)$ is as follows.

{\vskip 6pt \noindent \bf The Spin$^c$ manifold $M$.}
An interesting  and unusual feature of the $K$-cycle determined by a Heisenberg-elliptic operator on a contact manifold is that it consists of {\em two} components.

As a smooth manifold, $M$ is a disjoint union of two copies of $X\times S^1$,
\[ M=X^+\times S^1 \sqcup X^-\times S^1\]
Here $X^+$ and $X^-$ denote two copies of the contact manifold $X$ with opposite co-orientations.
A {\em co-orientation} of the contact structure $H$ is an orientation of the normal bundle $TX/H$.
 Each co-orientation of $X$ gives rise to a symplectic structure for $X\times S^1$
(one the opposite of the other), and therefore a Spin$^c$ structure for $X\times S^1$.

{\vskip 6pt \noindent \bf The complex vector bundle $E$.}
The $\CC$ vector bundle $E$ on $M$ is constructed from the principal Heisenberg symbol $\sigma_H(P)$ of $P$.
$E$ consists, of course, of two vector bundles: one on each copy of $X\times S^1$.
Let $\theta$ be a contact form on $X$ that agrees with the co-orientation on $X^+$,
and let $J$ be a complex structure of $H$ compatible with $d\theta$.
We denote by $H^{1,0}$ the $\CC$ vector bundle on $X^+$ with underlying $\RR$ vector bundle $H$ on which the scalar $i\in \CC$ acts as $J$,
and $H^{0,1}$ is the conjugate bundle.
If we choose a co-orientation for the contact structure of $X$,
then the Bargmann-Fock representation of the Heisenberg group
realizes the principal Heisenberg symbol $\sigma_H(P)$ as an invertible map between two (infinite rank) $\CC$ vector bundles on $X$ 
\[  \sigma_H^+(P)\;\colon\; F^0\otimes V^+\to F^1\otimes V^+\]
Here $V^+$ denotes the bundle of Fock spaces
\[ V^+=\bigoplus_{j=0}^\infty {\rm Sym}^j H^{1,0}.\]
Reversing the co-orientation of $X$ determines a second vector bundle isomorphism,
\[  \sigma_H^-(P)\;\colon\; F^0\otimes V^-\to F^1\otimes V^-\]
where $V^-$ is the bundle of conjugate Fock spaces
\[ V^-=\bigoplus_{j=0}^\infty {\rm Sym}^j H^{0,1}.\]
Composing $\sigma_H^\pm(P)$ with the isomorphism $\sigma_0^{-1}\colon F^1\to F^0$ extracted from the equatorial symbol $\sigma(P)|S(H^*)$, we obtain vector bundle {\em automorphisms} of $F^0\otimes V^\pm$,
and therefore elements in $K$-theory  $K^1(X)$,
\begin{align*}
[\sigma_0^{-1}\circ \sigma^+_H(P)]\in K^1(X)\\
[\sigma_0^{-1}\circ \sigma^-_H(P)]\in K^1(X)\\
\end{align*}
\begin{remark}
The vector bundles $F^0\otimes V^\pm$ are infinite direct sums of vector bundles of finite fiber dimension.
The automorphisms $\sigma_0^{-1}\circ \sigma^\pm_H(P)$, when compressed to a {\em finite} sum
$F^0\otimes \bigoplus_{j=0}^N \mathrm{Sym}^j H$, become stably equivalent (i.e., $K$-theoretically equivalent) for sufficiently large values of $N$.
Therefore, taking $N>>0$ sufficiently large $[\sigma_0^{-1}\circ \sigma^\pm_H(P)]$ are well-defined elements in $K$-theory.
In particular, the Chern characters of $[\sigma_0^{-1}\circ \sigma^\pm_H(P)]$ can be computed by the usual formalism taking $N>>0$ sufficiently large.
\end{remark}

The $\CC$ vector bundles $E^+$ and $E^-$ on $X\times S^1$ are constructed by a familiar clutching construction.
Quite generally, for a topological space $W$ with vector bundle $V\to W$, an automorphism 
$\sigma\;\colon\; V\to V$
determines a $\CC$ vector bundle $\nu(\sigma)$ on $W\times S^1$.
We set $S^1=S^1_+\cup S^1_-$,
where $S^1_+$ and $S^1_-$ are the upper and lower hemispheres of $S^1=\{|z|=1\}\subset \CC$.
Then $\nu(\sigma)$ is the $\CC$ vector bundle on $W\times S^1$ obtained as the quotient
\[ \nu(\sigma)=V\times S^1_+\cup_\sigma V\times S^1_-,\]
where the identification using $\sigma$ is
\[ (w,v, -1) \sim (w,v,-1),\quad
(w,v,1) \sim (w,\sigma(w)v,1)\qquad w\in W, v\in V_w, \pm 1\in S^1\]
Then on $X^+\times S^1$ we set
\[ E^+ = \nu(\sigma_0^{-1}\circ \sigma^+_H(P))\]
while on $X^-\times S^1$ we have
\[ E^- = \nu(\sigma_0^{-1}\circ \sigma^-_H(P))\]

{\vskip 6pt \noindent \bf The continuous map $\varphi$.}
The map $\varphi\,\colon M\to X$ is projection on the first factor $X$ for each of the two copies of $X\times S^1$.

\begin{theorem}\label{genthm}
Let $P$ be a Heisenberg-elliptic operator on a closed co-orientable contact manifold $X$,
\[ P\;\colon\; C^\infty(X,F^0)\to C^\infty(X,F^1)\]
Then the element in $K$-homology determined by $P$ is
\begin{equation*}
 [P]\;= \; [\sigma_0^{-1}\circ \sigma^+_H(P)]\;\cap\; [X^+]
\; + \; [\sigma_0^{-1}\circ \sigma^-_H(P)]\;\cap\; [X^-]
\end{equation*}
Equivalently, a $K$-cycle that represents $[P]$ in geometric $K$-homology is
\[ \mu^{-1}([P]) = (X^+\times S^1\sqcup X^-\times S^1,\, E^+\sqcup E^-,\, \varphi)\]
\end{theorem}

\subsection{Characteristic class formula}

If $X$ is a finite $CW$ complex and $\xi$ is an element in $KK^0(C(X), \CC)$,
\[ \xi \in KK^0(C(X),\CC),\]
then the homology Chern character of $\xi$ is an element ${\rm ch}(\xi)\in H_{ev}(X; \QQ)$ with the property: Whenever $F$ is a $\CC$ vector bundle on $X$,
\[ {\rm Index}\,(F\otimes \xi) = \epsilon_*({\rm ch}(F)\cap {\rm ch}(\xi)).\] 
Here $H_{ev}(X;\QQ)$ is the direct sum of the even rational homology groups of $X$,
\[ H_{ev}(X;\QQ) = H_0(X;\QQ)\oplus H_2(X;\QQ)\oplus H_4(X;\QQ)\oplus \cdots\]
and $\epsilon\,\colon X\to  \bullet$ is the map of $X$ to a point and $\epsilon_*\,\colon H_*(X;\QQ)\to H_*(\bullet; \QQ)=\QQ$ is the resulting map in rational homology.
As usual in algebraic topology ${\rm ch}$ is the Chern character and $\cap$ is cap product.

To explicitly construct a $K$-cycle $(M,E,\varphi)$ on $X$ with
\[ \mu(M,E,\varphi) = \xi\]
is to solve the index problem  for $\xi$ {\em integrally}. 
To solve the index problem for $\xi$ {\em rationally} is to give an explicit formula for ${\rm ch}(\xi)$.
If an $(M,E,\varphi)$ has been constructed, then
\[ {\rm ch}(\xi) = \varphi_*({\rm ch}(E)\cup {\rm Td}(M)\cap [M])\]
where $[M]$ is the fundamental cycle in $H_*(M;\ZZ)$ of $M$ 
and $\varphi_*\,\colon H_*(M;\QQ)\to H_*(X;\QQ)$ is the map of rational homology induced by $\varphi\,\colon M\to X$.

Thus, Theorem \ref{genthm} immediately implies the following characteristic class formula.
\begin{corollary}\label{thmchar}
The Fredholm index of a Heisenberg-elliptic operator $P$ on a closed contact manifold $X$,
\[ P\colon C^\infty(X,F^0)\to C^\infty(X,F^1)\]
is given by 
\begin{align*}
{\rm Index}\,P &= \int_X {\rm ch}(\sigma_0^{-1}\circ \sigma^+_H(P))\wedge {\rm Td}(H^{1,0})\\ 
&+ (-1)^{n+1} \int_X {\rm ch}(\sigma_0^{-1}\circ \sigma^-_H(P)) \wedge {\rm Td}(H^{0,1})
\end{align*}
Here $X$ is oriented (in both integrals) by the volume form $\theta(d\theta)^n$.
\end{corollary}
The factor $(-1)^{n+1}$ for the second term arises from the fact that $X^-\times S^1$ has the conjugate almost complex structure of $X^+\times S^1$,
and is of complex dimension $n+1$.

\begin{remark}
In several special cases, characteristic class formulas for the index of Heisenberg-elliptic operators have been derived before. 
A Toeplitz operator is (essentially) an order zero pseudodifferential operator in the Heisenberg calculus, and the Toeplitz index formula of Boutet de Monvel \cite{Bo79} is a special case of our formula.
We discuss this in detail in section \ref{Toeplitz}.

Epstein and Melrose prove a characteristic class formula for the index of the twisted sublaplacians $P_\gamma$,
and a separate formula for a class of operators they refer to as Hermite operators (a generalization of Toeplitz operators)  \cite{Epxx}.
Van Erp's index formula in \cite{vE10b} only applies in the special case of {\em scalar} Heisenberg-elliptic operators. 
Before now, no explicit result existed that was generally applicable and thus unified these various results.
In the absense of a general formula, the solution of the equivariant and families index problems was out of reach.
All previously obtained  formulas \citelist{\cite{Bo79} \cite{Epxx} \cite{vE10b}} are easily seen to be special cases of Corollary \ref{thmchar}. 
\end{remark}

\subsection{Examples}\label{examples}

Below are examples which illustrate that the geometric $K$-cycle of Theorem \ref{genthm} is explicitly computable and yields simple  clear formulas.

\begin{example}
The $K$-cycle for a second order Heisenberg-elliptic operator
\[P_\gamma=I_r\otimes \Delta_H + i\gamma(I_r\otimes T)\;\colon\; C^\infty(X,\CC^r)\to C^\infty(X,\CC^r)\]
is made explicit if we compute the principal Heisenberg symbol of $P_\gamma$ (see section \ref{final}).
The vector bundle $E^+\to X^+\times S^1$ in the $K$-cycle for $P_\gamma$ is constructed from the winding of $\gamma$
around the positive integers $n+2j$. For each $j=0,1,2,\dots$ the smooth map
\[ \gamma-(n+2j)\;\colon\; X\to GL(r,\CC)\]
determines an automorphism of the trivial vector bundle $X\times \CC^r$.
If $\nu$ denotes the clutching construction described above,
we obtain a $\CC$ vector bundle $\nu(\gamma-(n+2j)I_r)$ on $X^+\times S^1$. 
Then $E^+$ on $X^+\times S^1$ is the $\CC$ vector bundle
\[ E^+ = \bigoplus_{j=0}^N \;\nu(\gamma-(n+2j)I_r)\,\otimes\,    \varphi^* {\rm Sym}^j H^{1,0} \]
Likewise, $E^-$ on $X^-\times S^1$ is constructed from the winding of $\gamma$
around the negative integers $-(n+2j), j=0,1,2,\dots$,
\[ E^- = \bigoplus_{j=0}^N \;\nu(\gamma+(n+2j)I_r) \,\otimes\,\varphi^* {\rm Sym}^j H^{0,1}\]
\end{example}

\begin{example}
The odd-dimensional unit spheres $S^{2n+1}$ of $\RR^{2n+2}$ have standard contact form
\[ \theta = \sum_{i=1}^{n+1} x_{2i}dx_{2i-1}-x_{2i-1}dx_{2i} \]
where $x_1, x_2, \dots, x_{2n+2}$ are the standard coordinates on $\RR^{2n+2}$.
Theorem \ref{genthm} applies to give a simple and clear index formula for the $P_\gamma$ operators on $S^{2n+1}$. 
Any continuous map $f \colon S^{2n+1} \to GL(r, \CC)$ determines an element
in the homotopy group $\pi_{2n+1}(GL)$ of  $GL=\lim_{k \to \infty} GL(k,\CC)$. 
Bott periodicity identifies $\pi_{2n+1}(GL)$ with the integers, 
\[ \beta\;\colon\; \pi_{2n+1}(GL)\cong \ZZ.\]
Therefore an integer $\beta (f)$ has been assigned to $f$,  
\[ \beta(f) = ch(f)[S^{2n+1}] = \int_{S^{2n+1}} \mathrm{Tr}\left( \frac{f^{-1}df}{-2\pi i}\right)^{2n+1} \] 
The formula for the index of $P_\gamma$ is
\[
\mathrm{Index}\, P_\gamma = \sum_{j=0}^N 
\left(\begin{array}{c}n+j-1\\j\end{array}\right) 
\left[\beta(\gamma-(n+2j)I_r) +(-1)^{n+1}\beta(\gamma+(n+2j)I_r)\right]
\]
where $N$ is sufficiently large, as above. 
Observe that on an odd dimensional sphere any complex vector bundle is stably trivial.
Thus $H^{1,0}$ is stably trivial, and in the $K$-cycle for $P_\gamma$ we may replace $\varphi^*\mathrm{Sym}^j H^{1,0}$ by a trivial vector bundle of the same rank, which is $\left(\begin{array}{c}n+j-1\\j\end{array}\right)$.

In particular, on $S^3$ we obtain the very simple formula
\[
\mathrm{Index}\, P_\gamma = \sum_{k\,{\rm odd}} \beta(\gamma-kI_r)
\]
\end{example}

\section{Outline of the proof}\label{proofoutline}

In \cite{vE10a}  we showed that the principal Heisenberg symbol $\sigma_H(P)$ of a Heisenberg-elliptic operator $P$ determines an element in the $K$-theory group of a noncommutative $C^*$-algebra
\[ [\sigma_H(P)]\in K_0(C^*(T_HX))\]
$T_HX$ is the tangent bundle $TX$, where each fiber $T_xX=H_x\times \RR$ has the structure of a nilpotent Lie group isomorphic to the Heisenberg group.
$C^*(T_HX)$ is the convolution $C^*$-algebra of the groupoid $T_HX$.
The Connes-Thom isomorphism in analytic $K$-theory gives a canonical isomorphism
\[ \Psi\;\colon\; K^0(T^*X)\stackrel{\cong}{\longrightarrow} K_0(C^*(T_HX))\]

The proof of Theorem \ref{genthm} will be accomplished by constructing from the $K$-theory element $[\sigma_H(P)]$ the relevant $K$-cycle.
This will be done in three steps.

\vskip 6pt
\noindent {\bf Step 1}. In section \ref{NCT}, by a direct construction, an isomorphism of abelian groups 
\[ b\;\colon\; K_0(C^*(T_HX)) \to K^{top}_0(X)\]
will be defined such that there is commutativity in the triangle
\[ \xymatrix{   K^0(T^*X) \ar[dr]_c \ar[rr]^{\Psi} && K_0(C^*(T_HX))\ar[dl]^{b}\\
              & K_0^{top}(X) &
 }
\]
where $c$ is the standard Poincar\'e duality map.

\vskip 6pt
\noindent {\bf Step 2}. In section \ref{AnalyticTriangle} we prove that $b(\sigma_H(P))$ is the correct $K$-cycle, i.e.,
\[  \mu^{-1}([P]) = b(\sigma_H(P))\]
\vskip 6pt
\noindent {\bf Step 3}. Finally, in section \ref{computing_b}, $b(\sigma_H(P))$ is explicitly calculated, i.e., we prove
\[ b(\sigma_H(P)) = [\sigma^{-1}\circ \sigma^+_H(P)]\cap[X^+] + 
[\sigma^{-1}\circ \sigma^-_H(P)]\cap[X^-]\]

\section{Noncommutative topology of contact structures}\label{NCT}

In this section, the noncommutative Poincar\'e duality map
\[ b\;\colon\; K_0(C^*(T_HX))\to K^{top}_0(X)\]
will be defined.
$T_HX$ is the tangent bundle $TX$ with each fiber $T_xX$ viewed as a nilpotent Lie group.
If $\theta$ is a contact form on $X$, with contact hyperplane bundle $H:=\mathrm{Ker}\,\theta$,
then the tangent space $T_xX=H_x\times \RR$ at a point $x\in X$ has the structure of a nilpotent Lie group that is isomorphic to the Heisenberg group,
\[ (v,t)\cdot (v',t') = (v+v',t+t'-\frac{1}{2}d\theta(v,v'))\qquad v,v;\in H_x, t,t'\in \RR\]

The Type I $C^*$-algebra $C^*(T_HX)$ decomposes into stably commutative factors, i.e., there is short exact sequence of $C^*$-algebras
\[ 0\to I_H \to C^*(T_HX) \to C_0(H^*) \to 0\]
where $C_0(H^*)$ is commutative, while the ideal $I_H$ is Morita equivalent to a commutative $C^*$-algebra.
Moreover, this short exact sequence is the {\em quantization} of, and is $K$-theoretically equivalent to the short exact sequence of commutative $C^*$-algebras
\[ 0\to C_0(T^*X\setminus H^*)\to C_0(T^*X)\to C_0(H^*)\to 0\]
The definition and explicit calculation of $b$ is made possible by this reduction to commutative $C^*$-algebras.

\subsection{The $C^*$-algebra of the Heisenberg group}\label{group_algebra}

The $C^*$-algebra $C^*(T_HX)$ is a locally trivial bundle of $C^*$-algebras over $X$, whose fiber is the group $C^*$-algebra $C^*(G)$ of the Heisenberg group.
We start with a careful description of the structure of this algebra.
 
Let $V=\RR^{2n}$ be the standard symplectic vector space with symplectic form $\omega$. 
Multiplication in the Heisenberg group $G=V\times \RR$ is defined by
\[ (v,t)\cdot (v',t') = (v+v',t+t'+\frac{1}{2}\omega(v,v')).\]
The factor $\frac{1}{2}$ is chosen so that for two vectors $v,w\in V$ the Lie bracket is given by the symplectic form $[v,w]=\omega(v,w)$.

Fourier transform in the $\RR$ variable for functions $f\in C_c^\infty(G)$,
\[ \tilde{f}(v,s) = \int_\RR e^{-its}f(v,t)dt\]
completes to a $\ast$-isomorphism,
\[ C^*(G)\cong C^*(V\times \RR^*,c_\omega).\]
Here $V\times \RR^*$ denotes the vector bundle over $\RR^*$ with fiber $V$.
We may think of it as a smooth groupoid with object space $\RR^*$. 
The expression $c_\omega$ refers to the groupoid cocycle 
\[ c_\omega\;\colon\; V\times V\times \RR^*\to U(1)\;,\;c_\omega((v,s),(w,s)) = \exp{(is\omega(v,w))}.\]
The $C^*$-algebra $C^*(V\times \RR^*,c_\omega)$
is the completion of the groupoid convolution algebra $C_c^\infty(V\times \RR^*)$
twisted by the cocycle $c_\omega$, 
\[ (\tilde{f}\ast_{c} \tilde{g})(v,s) = \int_V e^{is\omega(v,w)} \tilde{f}(v-w,s)\tilde{g}(w,s)\,dw.\]
Elements in the $C^*$-algebra $C^*(G)$ can be identified with sections in a continuous field over $\RR^*$.
The fiber at $s\in \RR^*$ is the twisted convolution $C^*$-algebra $C^*(V,s\omega)$ 
(now with fixed value of the parameter $s$ appearing in the cocycle).
At $s=0$ the fiber is commutative, $C^*(V)\cong C_0(V^*)$.
If $s\ne 0$ the field is trivial,
and each fiber is isomorphic to the algebra of compact operators,
$C^*(V,s\omega)\cong \KK$.
Restriction to $s=0$ gives the decomposition
\[ 0\to C^*(V\times \RR^\times, c_\omega)\to C^*(G)\to C_0(V^*)\to 0,\]
A simple rescaling of the parameter $s\in \RR^\times=\RR^*\setminus \{0\}$ gives an isomorphism
\[ C^*(V\times \RR^\times, c_\omega)\cong C_0(-\infty,0)\otimes C^*(V,-\omega)\oplus C_0(0,\infty)\otimes C^*(V,\omega).\]
We will identify
\[ \RR^\times = (-\infty,0)\sqcup (0,\infty)\approx \RR\sqcup \RR\]
by the map $s\mapsto \log{|s|}$ for $s\in \RR^\times$.
This fixes  an isomorphism
\[ C^*(V\times \RR^\times, s\omega)\cong C_0(\RR)\otimes C^*(V,-\omega)\oplus C_0(\RR)\otimes C^*(V,\omega).\]
While both $C^*$-algebras $C^*(V,\pm\omega)$
are isomorphic to $\KK$, the algebra of compact operators on Hilbert space,  we do not identify these two algebras, for reasons that will soon become apparent.

\subsection{$U(n)$ symmetry}
A choice of contact 1-form $\theta$ and compatible complex structure $J$ in the fibers of $H$
amounts to a reduction of the structure group of $TX$ to $U(n)$.
For this reason it is useful to consider the $U(n)$ symmetry of the group $C^*$-algebra $C^*(G)$.

With the standard identification $V=\RR^{2n}=\CC^n$
the canonical action of the unitary group $U(n)$ on $V=\CC^n$
induces automorphisms of the Heisenberg group $G=V\times\RR$,
and therefore of the $C^*$-algebra $C^*(G)$.
Because the cocycle $c_\omega$ is invariant under the $U(n)$ action on $V$,
$U(n)$ also acts in the obvious way on the $C^*$-algebra $C^*(V\times \RR^*, c_\omega)$,

The $U(n)$ action on functions on $G=V\times \RR$ commutes with the Fourier transform in the $\RR$ variable.
Therefore the isomorphism
\[ C^*(G)\cong C^*(V\times \RR^*,c_\omega)\]
is $U(n)$ equivariant.
Likewise, the isomorphism 
\[ C^*(V\times\RR^\times, c_\omega) \cong C(\RR)\otimes C^*(V,\omega)\oplus C(\RR)\otimes C^*(V,-\omega)\]
is $U(n)$ equivariant.

Now choose a representation
\[ C^*(V,\omega)\cong \KK(\HH)\]
Up to unitary equivalence, this is a uniquely determined irreducible unitary representation of the Heisenberg group $G$. Below we will choose an explicit model for this representation.
But first we consider the $U(n)$ action on $\KK(\HH)$ induced by its action on $C^*(V,\omega)$.
The automorphism group of $\KK(\HH)$ is the projective unitary group $U(\HH)/U(1)$ of the separable Hilbert space $\HH$.
As we will see below, the  action of $U(n)$ on $\KK(\HH)$
lifts to a unitary representation of $U(n)$ on the Hilbert space $\HH$.
This is specific fact about the representation theory of the Heisenberg group is of crucial importance for our analysis,
because it  implies, later on, the vanishing of the Dixmier-Douady invariant of a specific continuous trace $C^*$-algebra over $X$,
providing us with an important Morita equivalence.

For the opposite algebra $C^*(V,-\omega)$
we have the isomorphism
\[ C^*(V,-\omega)\cong \KK(\overline{\HH}),\]
where $\overline{\HH}$ denotes the Hilbert space that is conjugate (or dual) to $\HH$, and that carries the dual representation of $G$ as well as of $U(n)$.
This isomorphism is equivariant with respect to the $U(n)$ actions,
and the Hilbert space $\overline{\HH}$ is {\em not} ismorphic to $\HH$ 
{\em as a representation space} of $U(n)$.

In summary, we obtain a $U(n)$ equivariant isomorphism
\[ C^*(V\times \RR^\times,c_\omega)
\cong C_0(\RR)\otimes \KK(\HH)
\oplus C_0(\RR)\otimes \KK(\overline{\HH}).\]
It will be useful to have a concrete model for the Hilbert space $\HH$
with its representations of the Heisenberg group $G$ and the unitary group $U(n)\subset {\rm Aut}(G)$.

\subsection{Bargmann-Fok space}

The representation space of the Heisenberg group $G$ that most clearly exhibits the $U(n)$ action is the Bargmann-Fok space.
As a $U(n)$ representation, the Bargmann-Fok space is a completion of the space of symmetric tensors of $\CC^n$,
\[ \HH^{BF} \supset {\rm Sym}\, \CC^n = \bigoplus_{j=0}^\infty {\rm Sym}^j\, \CC^n\]
with its standard $U(n)$ action.
If we identify symmetric tensors with complex polynomials on $\CC^n$
then the Bargmann-Fok space is the Hilbert space of entire functions on $\CC^n$ with inner product
\[ \ang{f,g} = \pi^{-n}\int f(z)\overline{g(z)} \,e^{-|z|^2} dz .\]
If $Z_1, \dots, Z_n$ is an orthonormal basis for $\CC^n$
then the monomials $Z_1^{m_1}\cdots Z_n^{m_n}/\sqrt{m_1!\cdots m_n!}$ form an orthonormal basis for $\HH^{BF}$.
In particular, the summands ${\rm Sym}^j\, \CC^n$ are mutually orthogonal.

Let us briefly review how the Heisenberg group $G$ acts on this space.
To describe the representation it is convenient to pass to the complexified Lie group $G_\CC$.
In order not to get confused, let us denote by $J\colon V\to V$ the standard complex structure on $V=\RR^{2n}$ and {\em not} identify $V$ with $\CC^n$ here.
The complexified space $V\otimes \CC$ splits into the $\pm \sqrt{-1}$ eigenspaces for $J$ as $V\otimes \CC = V^{1,0} \oplus V^{0,1}$.
Correspondingly, the complexification of the Lie algebra $\Lg$ of $G$ splits as a direct sum 
\[ \Lg \otimes \CC = V^{1,0}\oplus V^{0,1}\oplus \CC. \]
Recall that $\omega(Ju,Jv) = \omega(u,v)$ and $\omega(Jv,v)>0$ if $v\ne 0$.
If we extend the bilinear form $\omega$ to a complex bilinear form on $V\otimes \CC$ then the expression 
\[ \ang{z,w} := i\omega(z,\bar{w})\]
defines a hermitian inner product on $V^{1,0}$. 
Let $\{Z_1,\ldots,Z_n\}$ be a basis for $V^{1,0}$ that is orthonormal with respect to this hermitian product. 
For elements in the complexified Lie algebra $A, B\in \Lg\otimes \CC$ we have $[A,B]=\omega(A,B)=-i\ang{A,\bar{B}}$. Therefore
\[ [Z_j,\bar{Z}_k] =  \frac{1}{i} \delta_{jk} .\]
Since $V^{1,0}$ is orthogonal to $V^{0,1}$ we also obtain:
\[ [Z_j, Z_k] = [\bar{Z}_j, \bar{Z}_k] = 0.\]
In this picture the Bargmann-Fok space $\HH^{BF}$ is the completion of ${\rm Sym} V^{1,0}$, and can be identified with a space of entire functions on $V^{0,1}$. 
The representation $\pi$ of the complexified Lie algebra $\Lg\otimes\CC$ on $\HH^{BF}$ is given by
\[ \pi(Z_j) = iz_j \;,\; \pi(\bar{Z}_j) = i\frac{\partial}{\partial z_j} \;,\; \pi(1) = i .\]
The multiplication operator $z_j$ and complex derivative $\partial/\partial z_j$
are unbounded linear operators on $\HH^{BF}$.
The subspace of polynomials ${\rm Sym}\,V^{1,0}$ is an invariant subspace for both operators,
and hence also for the universal enveloping algebra $\UU(\Lg)$ of the Lie algebra.
The operators $z_j$ and $\partial/\partial z_j$ are adjoints, which implies that $\pi(W)^*=-\pi(\bar{W})$ for all elements $W\in \Lg\otimes \CC$.
Therefore $\pi(v)$ is skew adjoint for elements $v\in \Lg$ in the real Lie algebra,
and the corresponding representation for the group $G$ is unitary.


In summary, we have $U(n)$ equivariant $\ast$-homomorphisms
\[ \xymatrix{    C^*(G) \ar[d]_{f\mapsto\tilde{f}(-,1)}\ar[rd]^{\pi}\\ 
                 C^*(V,\omega) \ar[r]^{\cong}  
               & \KK(\HH^{BF}).
                }
\]
The horizontal arrow is the Bargmann integral transform, which assigns to a function on $V$ (or its Fourier transform on $V^*$) an operator on $\HH^{BF}$
\cite{Ba61}.
For the opposite algebra $C^*(V,-\omega)$ we obtain a representation on the conjugate space $\overline{\HH}^{BF}$,
which is a completion of the space of symmetric tensors of $V^{0,1}$
carrying the dual (or conjugate) representation of $U(n)$.
We will adopt the notation  
\[ \HH^{BF}_+ = \HH^{BF},\; \HH^{BF}_- = \overline{\HH}^{BF}\]

\begin{figure}[h!]
\centering
\includegraphics[width=12cm]{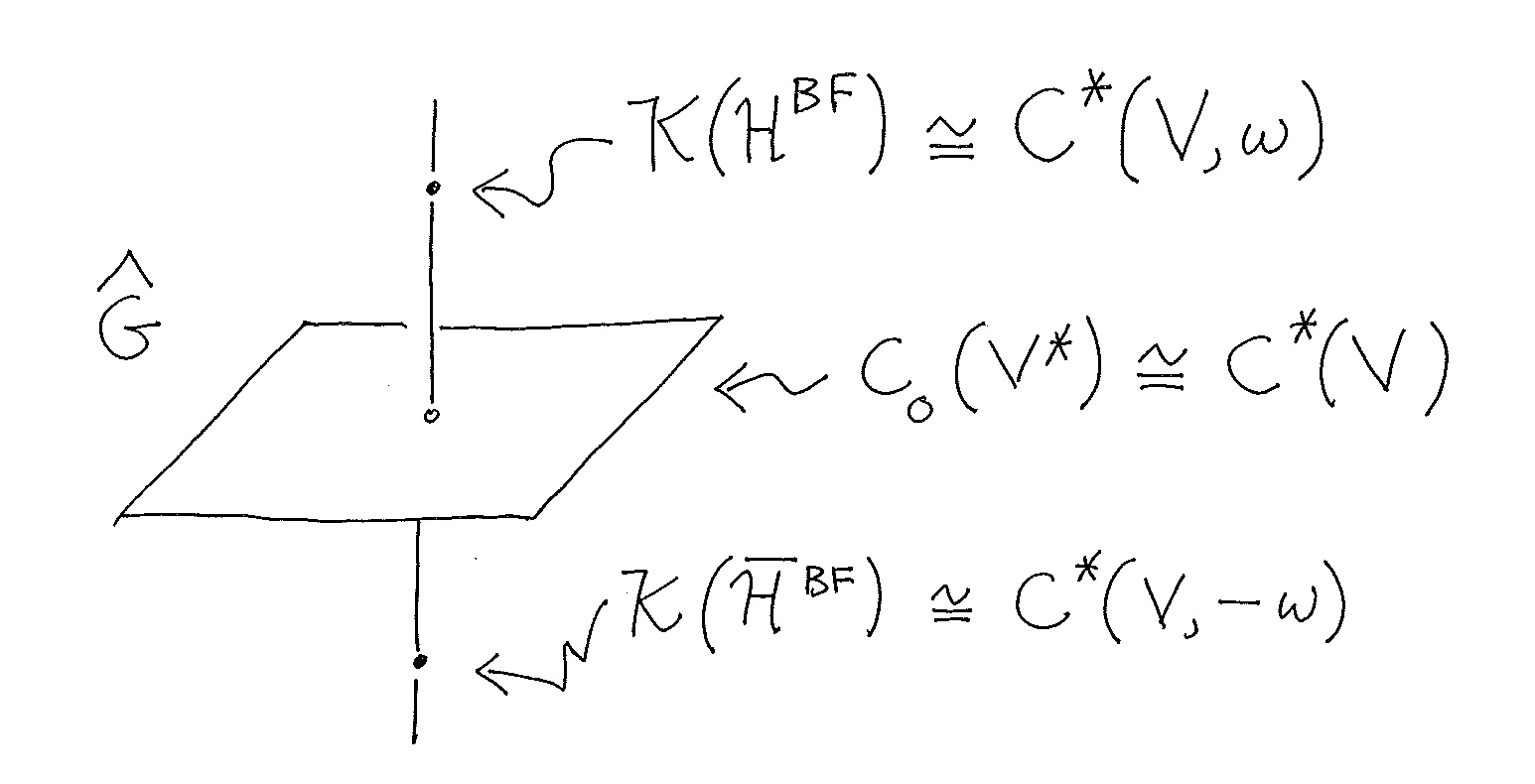}
\caption{The short exact sequence $0\to C_0(\RR^\times,\KK)\to C^*(G)\to C_0(V^*)\to 0$}
\end{figure}

\subsection{Quantization}

Recall that the group $C^*$-algebra $C^*(G)$ is a quantization of the commutative algebra $C_0(\Lg^*) = C_0(V^*\times \RR^*)$.
Consider the groupoid $V\times \RR^*\times [0,1]$, which is, algebraically, a disjoint union of copies of the abelian group $V$ parametrized by $(s,\varepsilon)\in \RR^*\times [0,1]$.
For this groupoid, define the groupoid $2$-cocycle
\[ c_{\varepsilon\omega}\;\colon\; V\times V\times \RR^*\times [0,1]\to U(1)\;,\;c_{\varepsilon\omega}((v,s,\varepsilon),(w,s,\varepsilon)) = \exp{(is\varepsilon\omega(v,w))}.\]
Then we have the twisted convolution $C^*$-algebra
\[ A = C^*(V\times \RR^*\times [0,1],c_{\varepsilon \omega}),\]
We can think of elements in this $C^*$-algebra $A$ as sections in a continuous field $\{A_\varepsilon\}$ 
with parameter $\varepsilon \in [0,1]$.
For $\varepsilon \in (0,1]$ the field is trivial, with $A_\varepsilon\cong A_1=C^*(G)$.
At $\varepsilon = 0$ we have the commutative algebra
\[ A_0 = C^*(V\times \RR^*)\cong C_0(V^*\times \RR^*).\]
We thus have a (strong deformation) quantization from $A_0$ to $A_1$.
As is well-known, this quantization induces an invertible $E$-theory element in $E(A_0,A_1)$
that, in turn, corresponds to an element in $KK(C_0(\Lg^*), C^*(G))$ that implements a $KK$-equivalence.

Now consider the decomposition discussed in section \ref{group_algebra}
\[ 0\to C^*(V\times \RR^\times, c_\omega) \to C^*(G)\to C^*(V)\to 0.\]
When restricted to the ideal
\[ C^*(V\times \RR^\times, c_\omega) 
\cong C_0(\RR)\otimes \KK(\HH^{BF}_-)
\oplus C_0(\RR)\otimes \KK(\HH^{BF}_+)
\]
the $KK$-equivalence $C^*(G)\sim C_0(\Lg^*)$ amounts to the {\em Morita equivalence}
\[ C_0(\RR)\otimes \KK(\HH^{BF}_-)
\oplus C_0(\RR)\otimes \KK(\HH^{BF}_+)\;
\sim\; 
C_0(\RR)\otimes \CC\oplus C_0(\RR)\otimes \CC\]
composed with the Bott isomorphisms $\beta_-\oplus \beta_+$, 
\[ \cdots\;
\sim\;C_0(\RR)\otimes C_0(V^*)\oplus C_0(\RR)\otimes C_0(V^*)\cong C_0(V^*\times \RR^\times).
\]
Here $\beta_+$ denotes the canonical Bott generator for the complex vector space $V^{1,0}= \CC^n$,
while $\beta_-$ denotes the Bott generator for the conjugate space $V^{0,1}$,
\[ \beta_+ \in KK(\CC, C_0(V^{1,0})),\;\;
\beta_- \in KK(\CC, C_0(V^{0,1})).\]
We obtain a $KK$-equivalence of short exact sequences:
\[ \xymatrix {
 0\ar[r]
 & C^*(V\times \RR^\times, c_\omega)\ar[r]\ar@{.>}[d]^{\cong}_{KK^{U(n)}}
 & C^*(G)\ar[r]\ar@{.>}[d]^{\cong}_{KK^{U(n)}}
 & C^*(V)\ar[r]\ar[d]^{\cong}  
 & 0
 \\
 0\ar[r]
 & C_0(V^*\times \RR^\times) \ar[r]
 & C_0(V^*\times \RR^*) \ar[r]
 & C_0(V^*)\ar[r]
 &0
 }
 \] 
Observe that all $\ast$-homomorphisms and $KK$-equivalences are $U(n)$-equivariant.
The diagram can therefore be interpreted as a commutative diagram in the category $KK^{U(n)}$.

\subsection{Associated bundles}
Our entire analysis  carries over to associated bundles (of groups, Hilbert spaces, $C^*$-algebras) over a contact manifold $X$.
As above, the contact 1-form $\theta$ and complex structure $J$ in the fibers of $H$
induce a reduction of the structure group of the tangent bundle $TX$ to $U(n)$.
Let $P_U$ denote the principal $U(n)$ bundle of orthonormal frames in $H^{1,0}$ with respect to its hermitian structure associated to $\theta, J$.

If $\alpha$ denotes the standard action $\alpha$ of $U(n)$ on $\CC^n$ then
\[ H^{1,0} = P_U\times_\alpha \CC^n\]
Likewise, the tangent space $TX$ can identified with
\[ TX = P_U \times_{\alpha\oplus 1} (\CC^n\oplus \RR) = H^{1,0}\oplus \underline{\RR}\]
This identification exhibits the stably almost complex (and hence Spin$^c$) structure of $TX$.
The bundle of Heisenberg groups $T_HX$ identifies with
\[ T_HX = P_U \times_{\alpha\times 1} G\]
Here $\alpha\times 1$ denotes the $U(n)$ action on $G=\CC^n\times \RR$.
The induced action on $C^*(G)$ gives the convolution $C^*$-algebra $C^*(T_HX)$ of the groupoid $T_HX$,
\[ C^*(T_HX) = P_U\times_{\alpha\times 1} \,C^*(G)\]
Denote by $\rho$ the action of $U(n)$ on the Bargmann-Fok space $\HH^{BF}=\HH^{BF}_+$,
the completion of the symmetric tensors ${\rm Sym}\,\CC^n$ described above.
We can form the associated bundle of Hilbert spaces 
\[ V^{BF}_+ = P_U\times_\rho \HH^{BF}_+\]
Continuous sections in $V^{BF}_+$ form a Hilbert module over $C(X)$,
and we have the dense subspace
\[ \bigoplus_{j=0}^\infty {\rm Sym}^jH^{1,0}\subset V^{BF}_+\]
The conjugate representation to $\rho$ gives rise to the dual module $V^{BF}_-$ with fiber $\HH^{BF}_-$ and
\[ \bigoplus_{j=0}^\infty {\rm Sym}^jH^{0,1}\subset V^{BF}_-\]
If $\pi$ denotes the Bargmann-Fok representation of $G$ on $\HH^{BF}$,
then for fixed $u\in U(n)$ the operator $\rho(u)\in U(\HH^{BF})$ is an intertwiner of the representations $\pi$ and $\pi\circ (\alpha\times 1)(u)$ of $G$ for the automorphism $(\alpha\times 1)(u)$  of $G$,
\[ \pi((\alpha\times 1)(u).g) = \rho(u)\pi(g)\rho(u)^{-1}\]
This compatibility of $\alpha\times 1$, $\rho$ and $\pi$ (respectively: $U(n)$ acting on $G$, $U(n)$ acting on $\HH^{BF}$, and $G$ acting on $\HH^{BF}$) implies that $\pi$
is well-defined as a representation of a fiber $C^*(G_x)$ of the bundle $C^*(T_HX)$ on the fiber $V_x^{BF}$ of the Hilbert module $V^{BF}_+$.
The representation $\pi$ therefore induces a $\ast$-homomorphism of $C(X)$-algebras,
\[ \pi\;\colon\;C^*(T_HX)\to \KK(V^{BF}_+)\]
Here $\KK(V^{BF}_+)$ denotes the `compact operators' on the $C(X)$ Hilbert module $V^{BF}_+$. 
Our analysis of the structure of $C^*(G)$ carries over to $C^*(T_HX)$.
We obtain a short exact sequence
\[ 0\to I_H \to C^*(T_HX) \to C_0(H^*) \to 0\]
where the ideal $I_H$ can be identified with
\begin{align*}
I_H &\cong 
P_U \times_{1\otimes Ad(\overline{\rho})} [C_0(\RR)\otimes \KK(\HH^{BF}_-)]
\;\oplus\; P_U \times_{1\otimes Ad(\rho)} [C_0(\RR)\otimes \KK(\HH^{BF}_+)]\\
&\cong C_0(\RR)\otimes \KK(V^{BF}_-)
\oplus C_0(\RR)\otimes \KK(V^{BF}_+).
\end{align*}
In particular, the ideal  $I_H$ is Morita equivalent to the commutative algebra $C_0(X\times \RR^\times)$,
and we have an explicit imprimitivity bimodule,
namely the disjoint union of the pull-back of $V^{BF}_+$ to $X\times (0,\infty)$
and the pull-back of $V^{BF}_-$ to $X\times (-\infty,0)$.

\subsection{Inverting the Connes-Thom isomorphism}
In \cite{vE10a} we discussed the crucial role of the isomorphism
\[ \Psi\;\colon\; K^0(T^*X)\to K_0(C^*(T_HX))\]
for our index problem.
The isomorphism $\Psi$ is essentially the Connes-Thom isomorphism in each fiber,
which agrees with the $KK$-equivalence induced by the quantization from $C_0(\Lg^*)$ to $C^*(G)$.
We now obtain a better grip on this isomorphism by analyzing how the quantization behaves when we decompose $C^*(T_HX)$.
Our analysis of $C^*(G)$ carries over  to $C^*(T_HX)$ because of the $U(n)$ equivariance of all the relevant constructions.

Quite generally, a principal $G$ bundle $P$ for a compact group $G$ and compact base $X=P/G$ induces a functor from the category $KK^{G}$ (equivariant $KK$-theory) to the category $KK^{X}$ ($RKK$-theory for $C(X)$-algebras) in the obvious way.
The functor assigns to a $G$-$C^*$-algebra $A$ the $C(X)$-$C^*$-algebra of $G$-equivariant continuous functions 
\[ P(A) = \{ f\;\colon\; P\to A\;\mid\; f\;\text{\rm continuous},\;f(pg)=g^{-1}f(p)\}.\]
Similarly for the morphisms: to a $G$ Hilbert module $\EE$ over $A$ 
the functor assigns the $C(X)$ Hilbert module of $G$-equivariant continuous functions $P\to \EE$, etc.

In this way the $U(n)$-equivariant $KK$-equivalence of short exact sequences for $C^*(G)$ that we derived above implies $KK$-equivalence of the associated sequences for bundles over $X$.
We obtain
\[ \xymatrix {
 0\ar[r]
 & I_H\ar[r]\ar@{.>}[d]^{\cong}_{KK^X}
 & C^*(T_HX)\ar[r]\ar@{.>}[d]^{\cong}_{KK^X}
 & C^*(H)\ar[r]\ar[d]^{\cong}  
 & 0
 \\
 0\ar[r]
 & C_0(H^*\times \RR^\times) \ar[r]
 & C_0(H^*\times \RR^*) \ar[r]
 & C_0(H^*)\ar[r]
 &0.
 }
 \] 
Observe that the bottom sequence is just
\[ 0\to C_0(T^*X\setminus H^*)\to C_0(T^*X)\to C_0(H^*)\to 0,\]
induced by the inclusion $H^*\subset T^*X$.
Following our analysis of $C^*(G)$ we see that the $KK$-equivalence $I_H\sim C_0(H^*\times \RR^\times)$ is a composition of the isomorphism of $C^*$-algebras
\[ I_H \cong C^*(H\times \RR^\times, c_\omega) 
\cong C_0(\RR)\otimes \KK(V^{BF}_-)
\oplus C_0(\RR)\otimes \KK(V^{BF}_+)
\]
with the Morita equivalence
\[  \cdots \;
\sim\; 
C_0(X\times \RR)\oplus C_0(X\times \RR)\]
and, finally, the two Thom isomorphisms $\tau^-\oplus \tau^+$, 
\[ \cdots\;
\sim\;C_0(H^{0,1}\times \RR)\oplus C_0(H^{1,0}\times \RR)\cong C_0(H^*\times \RR^\times).
\]
Here $\tau^+$ and $\tau^-$ denote the Thom classes for the complex bundles $H^{1,0}$ and $H^{0,1}$ respectively.

Let us isolate the maps in $K$-theory that are relevant for our purposes.
We summarize the conclusion of our analysis in the form of a proposition.
\begin{proposition}\label{Crux}
The following diagram commutes
\[ \xymatrix{   K^0(T^*X\setminus H^*) \ar[r] 
              & K^0(T^*X)\ar[dd]_{\Psi}^{Connes-Thom} \\
                K^0(X\times \RR^\times) \ar[u]^{\cup \tau^\pm}_{\text{Thom isom.}}
              & \\
                K_0(I_H) \ar[u]_{\text{Morita equiv.}}^{\otimes V^{BF}_\pm}\ar[r]
              & K^0(C^*(T_HX))    
                }
\]
All vertical maps in the diagram are isomorphisms, and the horizontal maps (induced by inclusions) are surjective.
The vertical arrows on the left are explicitly given by
\[ \xymatrix{   K^0(H^{0,1}\times \RR)\oplus K^0(H^{1,0}\times \RR)  \ar[r]^-{\cong}
              & K^0(T^*X\setminus H^*) \\ 
                K^0(X\times \RR)\oplus K^0(X\times \RR)\ar[r]^-{\cong}\ar[u]^{\cup \tau^-\oplus \cup \tau^+}_{\text{Thom isom.}}
              & K^0(X\times \RR^\times) \ar[u] \\
                K_0(C_0(\RR)\otimes \KK(V^{BF}_-))\;\oplus\; K_0(C_0(\RR)\otimes \KK(V^{BF}_+))\ar[r]^-{\cong}\ar[u]_{\text{Morita equiv.}}^{\otimes V^{BF}_-\oplus \;\otimes V^{BF}_+}
              & K_0(I_H) \ar[u] 
                }
\]
\end{proposition}
\begin{proof}
We only need to check the claim about surjectivity of the maps represented by the horizontal arrows.
The inclusion of each connected component of $T^*X\setminus H^*$ into $T^*X$ is a homotopy equivalence.
Therefore $K^0(T^*X\setminus H^*)\to K^0(T^*X)$ is surjective.
Since all vertical maps are isomorphisms, the same is true for the map $K_0(I_H)\to K^*(C^*(T_HX))$.
\end{proof}

\subsection{Noncommutative Poincar\'e duality}\label{PD_geometric}

\begin{definition}\label{map_b}
The {\em noncommutative Poincar\'e duality} map 
\[ b\;\colon K_0(C^*(T_HX))\to K^{top}_0(X)\]
is defined by choosing an arbitrary lift of an element in $K_0(C^*(T_HX))$ to an element in $K_0(I_H)$, followed by the composition of maps
\[ K_0(I_H)\cong K^1(X)\oplus K^1(X) \stackrel{\cap [X^-]\oplus \cap [X^+]}{\longrightarrow} K^{top}_0(X)\oplus K^{top}_0(X)\to K^{top}_0(X).\]
From left to right, these maps are (1) Morita equivalence induced by the Bargmann-Fok Hilbert modules $V^{BF}_-$ and $V^{BF}_+$, (2) Poincar\'e duality for the two natural Spin$^c$ structures on $X$ and (3) addition of $K$-homology classes.
\end{definition}
\begin{remark}
Let us repeat here that the implicit isomorphism
\[ K^0(X\times \RR^\times)\cong K^0(X\times \RR)\oplus K^0(X\times \RR) \cong K^1(X)\oplus K^1(X)\]
is chosen such that the identifications $(-\infty,0)\approx \RR$ and $(0,\infty)\approx \RR$ are given by the map $s\mapsto \log{|s|}$,
i.e., each component of $\RR^\times$ is oriented from $0$ to $\pm\infty$.
\end{remark}
\begin{remark}
The $K$-theory class of the symbol $[\sigma_H(P)]\in K_0(C^*(T_HX))$ only depends on the principal symbol of $P$ in the Heisenberg calculus, and is uniquely defined by $P$, independent of arbitrary choices. 
In section \ref{computing_b} we will see that lifting $[\sigma_H(P)]$ to $K_0(I_H)$ amounts to taking a distribution $\sigma_H(P)\in \EE'(T_HX)$ that represents a {\em full} Heisenberg symbol for $P$, and possibly perturbing it by a compactly supported smooth function in $C_c^\infty(T_HX)$ in order to give it an extra property.
Such a perturbation does not affect the {\em principal} Heisenberg symbol.
The choice of such a perturbation is always possible, but not uniquely determined by $P$.
However, the equivalence class of $b(\sigma_H(P))$ in the $K$-homology group $K^{top}_0(X)$ does not depend on the choice of this perturbation,
and so the map $b$ is well-defined.
The proof of this fact is implicit in the proof of the following theorem.
\end{remark}

\begin{lemma}\label{Dirac}
The following diagram commutes,
\[ \xymatrix{    K^0(X\times \RR)\ar[r]^-{\cup\tau_+}\ar[d]^{\cong}
               & K^0(T^*X) \ar[d]^{c}\\
                 K^1(X) \ar[r]_{\cap [X^+]}
               & K^{top}_0(X) 
                }
\]
Here the left vertical arrow is the suspension isomorphism in $K$-theory;
the upper horizontal arrow is the Thom isomorphism resulting from the direct sum decomposition $T^*X = H^{1,0}\oplus \underline{\RR}$;
the lower horizontal arrow is Poincar\'e duality on the Spin$^c$ manifold $X$,
i.e., is cap product with the $K$-homology fundamental cycle $[X^+]$;
and the right vertical arrow is the ``clutching construction'', i.e., cap product with the $K$-homology fundamental cycle of the Spin$^c$ manifold $T^*X$.
\end{lemma}
\begin{proof}
Commutativity of the diagram follows from commutativity of the two triangles that appear  if we introduce the diagonal arrow which is the Thom isomorphism for the Spin$^c$ vector bundle $T^*X$ on $X$,
\[ \xymatrix{    K^0(X\times \RR)\ar[r]^-{\cup\tau_+}
               & K^0(T^*X) \ar[d]^{c}\\
                 K^1(X) \ar[r]_{\cap [X^+]} \ar[ur]\ar[u]^{\cong}
               & K^{top}_0(X) 
                }
\]
Commutativity of the two triangles is standard algebraic topology, which we now briefly indicate.
The upper triangle commutes because the Thom class for $T^*X = H^{1,0}\oplus\underline{\RR}$ is the product of the Thom classes used for the left vertical arrow and the upper horizontal arrow.
For commutativity of the lower triangle, use the following notation: 
$\beta \in KK^1(C(X), C_0(T^*X))$ is the Thom class for $T^*X$,
$\alpha\in KK^1(C_0(T^*X), C(X))$ is the $KK$-element given by the family of Dirac operators for the fibers of $T^*X$.
The Dirac-dual Dirac identity in this context is the assertion that the Kasparov product $\beta\# \alpha$ is the unit element of the ring $KK^0(C(X), C(X))$,
\[ \beta\# \alpha = 1\in KK^0(C(X), C(X)).\]
The total space of $T^*X$ is itself a (non-compact) Spin$^c$ manifold,
and its $K$-homology fundamental class $[T^*X]$ is the Kasparov product of $\alpha$ and $[X^+]$,
\[ [T^*X] = \alpha\# [X^+] \in KK^0(C_0(T^*X), \CC).\]
It now follows that 
\[ \beta\# [T^*X]= \beta\# \alpha \# [X^+] = [X^+].\]
which gives commutativity of the lower triangle.
\end{proof}

\begin{theorem}\label{Poincare}
The following diagram commutes
\[ \xymatrix{    K^0(T^*X) \ar[r]^-{\Psi}\ar[d]^{c} 
               & K_0(C^*(T_HX)) \ar[ld]^{b}  \\
                 K_0^{top}(X)
               &  }
\]
where $b$ is as in Definition \ref{map_b},
and $c$ is the ``clutching construction'' introduced in \cite[section 22]{BD80}.
\end{theorem}
\begin{proof}
If we identify $X\times (0,\infty)\cong X\times \RR$, then  Lemma \ref{Dirac} gives commutativity of 
\[ \xymatrix{    K^0(X\times (0,\infty))\ar[r]^-{\cup\tau_+}\ar[d]^{\cong}
               & K^0(T^*X) \ar[d]^{c}\\
                 K^1(X) \ar[r]_{\cap [X^+]}
               & K^{top}_0(X) 
                }
\]
Similarly, for the conjugate Spin$^c$ structure on $X$ we have
\[ \xymatrix{    K^0(X\times (-\infty, 0))\ar[r]^-{\cup\tau_-}\ar[d]^{\cong}
               & K^0(T^*X) \ar[d]^{c}\\
                 K^1(X) \ar[r]_{\cap [X^-]}
               & K^{top}_0(X) 
                }
\]
The {\em sign} of the isomorphism $K^0(X\times (0,\infty))\cong K^1(X)$ depends on the orientation of $(0,\infty)$, i.e., on the orientation of the normal bundle $N=X\times \RR$,
which in turn depended on the choice of contact form.
In order to make the second diagram commute we must orient $(-\infty,0)$ from $0$ to $-\infty$.
This is because the reversed orientation of the normal bundle $N$ is built into the definition of the fundamental cycle $[X^-]$, and also affects the choice of the isomorphism
\[ K^0(X\times (-\infty, 0))\cong K^0(X\times \RR)\cong K^1(X).\]  
Combining the two components of $X\times \RR^\times$, we obtain a single diagram
\[ \xymatrix{    K^0(X\times (-\infty,0))\oplus K^0(X\times (0,+\infty))\ar[r]^-{\cup\tau^-\oplus \cup\tau^+}\ar[d]^{\cong}
               & K^0(T^*X\setminus H^*) \ar[r]
               & K^0(T^*X) \ar[d]^{c}\\
                 K^1(X)\oplus K^1(X) \ar[r]_{\cap [X^-]\oplus \cap [X^+]}
               & K^{top}_0(X) \oplus K^{top}_0(X)\ar[r]
               & K^{top}_0(X) 
                }
\]
Comparing this diagram with Corollary \ref{Crux}  proves the proposition.

\end{proof}

\section{The index theorem as a commutative triangle}\label{AnalyticTriangle}

The symbol $\sigma_H(P)$ of a hypoelliptic operator in the Heisenberg calculus naturally determines an element in $K_0(C^*(T_HX))$.
In this section we prove that the Poincar\'e dual $b(\sigma_H(P))\in K_0^{top}(X)$ of the Heisenberg symbol is the desired $K$-cycle $\mu^{-1}(P)$
\[ \mu^{-1}(P) = b(\sigma_H(P))\]
For convenience of the reader, we start with a brief sketch of the main features of the Heisenberg calculus.
References are \citelist{\cite{Ta84} \cite{BG88} \cite{Ep04}}.

\subsection{The Heisenberg filtration}

Consider a differential operator $P$ on a smooth manifold $X$,
given in local coordinates by an expression
\[ P = \sum_{|\alpha|\le d} a_{\alpha} \partial^\alpha, \qquad \alpha=(\alpha_1, \dots, \alpha_n),\; |\alpha| = \sum \alpha_j,\; \partial^\alpha = \Pi (\partial/\partial x_j)^{\alpha_j},\]
where the coefficients $a_\alpha$ are smooth functions.
The highest order part of $P$ is not well-defined as a differential operator on $X$.
The algebra of differential operators is only filtered, not graded.
But the highest order part at a point $x\in X$,
\[ P_x = \sum_{|\alpha|=d} a_{\alpha}(x) \partial^\alpha,\]
can be interpreted as a constant coefficient operator on the tangent fiber $T_xX$. 
As such it is well-defined and independent of coordinate choices. 

The root of the {\em Heisenberg calculus} is a simple but important idea, 
proposed by Gerald Folland and Elias Stein in \cite{FS74},
to equip the algebra of differential operators with an alternative filtration.
The filtration on the algebra of differential operators proposed by Folland and Stein is generated by a filtration on the Lie algebra of vector fields, defined as follows: 
All vector fields in the direction of the contact hyperplane bundle $H$ have order one, as usual, but {\em any} vector field that is not everywhere tangent to $H$ is given order {\em two}. 
For example, a sublaplacian $\Delta_H$ has order two, as in the classical calculus.
But in the Heisenberg calculus the Reeb vector field $T$ is also a second order operator. 

What sort of object is the {\em highest order part} of a differential operator if we adopt this Heisenberg filtration?
Abstractly, for any filtered algebra, the notion of `highest order part' refers to an element in the associated graded algebra.
Observe that in the associated graded algebra smooth functions $f$ commute with all vector fields $W$, because the commutator $[W,f]=W.f$ is of order zero. 
This implies that elements in the graded algebra can be {\em localized} at points $x\in X$.
Therefore, just as in the classical case, the highest order part of a differential operator $P$ will consist of a family of operators $P_x$ parametrized by $x\in X$.
But they are not exacty constant coefficient operators here.
Instead, they are translation invariant operators for a certain nilpotent group structure on $T_xX$.

In all that follows, $X$ denotes a closed contact manifold with a contact hyperplane bundle $H\subset TX$.
We denote by $N=TX/H$ the quotient line bundle, and assume that there exists a global contact form $\theta$, so that we can identify $N=X\times \RR$.
Also, the Reeb vector field provides us with a section $N\subset TX$,
so that we can identify $T_xX=H_x\times \RR$.
Also, $d\theta$ restricted to $H_x$ is a symplectic form
and the tangent space $G_x=H_x\times \RR$ is then a Heisenberg group with group operation
\[ (v,t)\cdot (v', t') = (v+v', t+t' -\frac{1}{2}d\theta(v,v')),\qquad v,v'\in H_x, \; t, t'\in \RR.\]
Algebraically, the smooth groupoid $T_HX$ is the disjoint union of Heisenberg groups  
\[ T_HX = \bigsqcup_{x\in X} G_x.\]
The highest order part $\{P_x, x\in X\}$ of a differential operator $P$ in the Heisenberg calculus can be interpreted as a smooth family of translation invariant operators $P_x$ on the Heisenberg groups
$G_x=H_x\times \RR \cong T_xX$,
or, equivalently, a right invariant differential operator on the Lie groupoid $T_HX$.

For example, for the second order operators $P_\gamma = \Delta_H+i\gamma T$ that are locally represented as
\[ P = -\sum_{j=1}^{2n} W_j^2 + i\gamma T\]
freezing coefficients at a point $x\in X$ results in
\[ P_x = -\sum_{j=1}^{2n} W_j(x)^2 + i\gamma(x) T(x)\]
This formal polynomial in the tangent vectors $W_j(x)$, $T(x)\in T_xX$ should be interpreted as an element in the universal enveloping algebra $\UU(\Lg_x)$ of the Heisenberg Lie algebra $\Lg_x= H_x\oplus \RR$,
or, equivalently, as an invariant differential operator on the group $G_x$. 
In other words, the vector $W_j(x)$ should {\em not} be identified with a vector field on $T_xX$ that is translation invariant for the usual vector space structure on $T_xX$, but rather with a vector field on $T_xX$ that is translation invariant for the Heisenberg group structure on $T_xX$.

\subsection{Heisenberg pseudodifferential calculus}

When we restrict attention to differential operators
the Heisenberg calculus is fairly straightforward.
Constructing the corresponding $\ZZ$-filtered algebra of Heisenberg {\em pseudodifferential} operators requires more work (see \citelist{\cite{Ta84} \cite{BG88} \cite{Ep04}}).
In this section we sketch one possible approach.

Just like the classical pseudodifferential calculus, the Heisenberg algebra consists of pseudolocal continuous linear operators 
\[ P\;\colon\; C^\infty(X)\to C^\infty(X). \]
In other words, they are operators with a Schwartz kernel $k(x,y)$ that is a smooth function off the diagonal in $X\times X$.
The Heisenberg calculus (as well as its filtration) is defined by asymptotic expansions of the singularity of $k$ in the direction transversal to the diagonal.

Choose a connection $\nabla$ on $TX$ that preserves the distribution $H$
(this is an important technical condition).
Consider the exponential map ${\exp}^\nabla$ associated to this connection $\nabla$,
\[ h\;\colon\; T_HX\to X\times X\;\colon\; (x, v)\mapsto ({\rm exp}^\nabla_x(v), x).\]
The map $h$ is a local diffeomorphism of a neighborhood of the zero section of $T_HX$ with a neighborhood of the diagonal in $X\times X$.
We pull back the distribution $k$ to the groupoid $T_HX$ by the map $h$,
and then chop it by a cut-off function.
Let $\phi$ be an arbitrary smooth function on $T_HX$
that is compactly supported and equals $1$ in a neighborhood of the zero section.
At a point $x\in X$ let $k_x$ be the distribution on the Heisenberg group $G_x=T_xX$ defined by
\[ k_x(v) = \phi(x,v) \cdot k(h(x, -v)) , \qquad v\in T_xX.\]
The smooth family of distributions $\{k_x, x\in X\}$ defines a compactly supported distribution on $T_HX$. 
The action of the operator $P$ with kernel $k$
is approximated (in a neighborhood of the point $x\in X$) by convolution with the compactly supported distribution $k_x$ on the Heisenberg group $G_x$,
in a sense that can be made precise.
This was the basic idea introduced by Folland and Stein in \cite{FS74}.
 
Because the kernel $k$ is smooth off the diagonal, each distribution $k_x$ is regular, i.e., it is a smooth function when restricted to $G_x\setminus \{0\}$.
We say that $P$ is a pseudodifferential operator of order $d$ if
each distribution $k_x$ has an asymptotic expansion near $0\in G_x$,
\[ k_x \sim k_x^0+k_x^1+k_x^2+\cdots\]
The asymptotic expansion should be interpreted in the usual way: the remainder $k-\sum_{j=0}^N k^j$ becomes more regular as $N$ increases,
and the entire expansion determines $k_x$ modulo compactly supported smooth functions $C_c^\infty(T_HX)$.
The defining feature of the Heisenberg calculus is that the terms $k_x^j$ in this expansion must be homogeneous with respect to the `parabolic' dilation structure
of the Heisenberg group $G_x=T_xX=H_x\times \RR$,
\[\delta_s\;\colon\; H_x\times \RR \to H_x\times \RR\;\colon\; (v,t)\mapsto (sv, s^2t),\qquad s>0.\]
Formally, the term $k_x^j$ in the expansion must satisfy 
\[ k_x^j(\delta_s v) = s^{j-(2n+2)}\,k_x^j(v).\]
This notion of homogeneity based on the dilations $\delta_s$ 
corresponds to the Heisenberg filtration for differential operators that assigns order one to vector fields tangent to $H$ and order two to vector fields transveral to $H$.
But we now obtain a filtered algebra of pseudodifferential operators that differs from the classical pseudodifferential algebra.

The existence of an asymptotic expansion is independent of the choice of exponential map $\exp^\nabla$ and cut-off function $\phi$,
as long as the exponential map satisfies the technical condition that it preserves sections in $H$.
We denote the operator of convolution with the distribution $k_x$ by $P_x$,
and we write
\[ \sigma_H(P) = \{P_x, x\in X\}.\]
We regard $\sigma_H(P)$ as an element in the convolution algebra $\EE'(T_HX)$ of compactly supported distributions on the groupoid $T_HX$,
and we can think of it as the `full symbol' of $P$.
While the value of $\sigma_H(P)$ depends on the choice of exponential $\exp^\nabla$ and cut-off $\phi$,
the highest order part $\sigma^d_H(P) = \{k^0_x, x\in X\}$ in the asymptotic expansion of $\sigma_H(P)$ is invariantly defined as a smooth family of convolution operators on $T_HX$, independent of $\exp^\nabla, \phi$.

The collection of all Heisenberg pseudodifferential operators on $X$ is a $\ZZ$-filtered algebra: if $P, Q$ are Heisenberg pseudodifferential operators of order $a, b$ respectively, then $PQ$ is a Heisenberg pseudodifferential operator of order $a+b$.
Moreover, the leading term in the asymptotic expansion of $\sigma_H(PQ)$
is obtained by convolution of the leading terms in the expansions of $\sigma_H(P)$ and $\sigma_H(Q)$, i.e.,
\[ \sigma^{a+b}_H(PQ) = \sigma^a_H(P)\,\ast\,\sigma_H^b(Q).\]
As usual,
if the leading term $\sigma_H^d(P)$ of an order $d$ symbol $\sigma_H(P)$ is invertible (in the convolution algebra),
then we can derive an asymptotic expansion for an inverse of the full symbol $\sigma_H(P)$ modulo $C_c^\infty(T_HX)$.
This implies that $P$ itself is invertible in the Heisenberg algebra modulo smoothing operators, and hence a hypoelliptic Fredholm operator on $X$.
It is precisely for such operators that we prove our index formula.

\subsection{The Heisenberg symbol in $K$-theory}\label{K-symbol}

In \cite[Definition 15]{vE10a} we constructed a $K$-theoretic symbol in $K_0(C^*(T_HX))$ for a hypoelliptic differential operator $P$ in the Heisenberg calculus.
As we have seen, the principal symbol $\sigma^d_H(P)$ is then a family of differential operators $P_x$ on $G_x$, where each $P_x$ is obtained from $P$ by a simple procedure of `freezing coefficients' at $x\in X$.
The result is independent of a choice of exponential map $T_HX\to X\times X$ or cut-off function $\phi$.
In fact, the construction given in \cite{vE10a} makes no reference to the Heisenberg pseudodifferential calculus at all.
In \cite{vE10b} we gave an alternative definition that works for order zero operators in the Heisenberg calculus.

Perhaps the easiest way to define the $K$-theory element 
\[ [\sigma_H(P)]\in K_0(C^*(T_HX))\]
is to follow the general ideas set out by Connes in \cite[II.9.$\alpha$]{Co94}.
The resulting construction works for arbitrary hypoelliptic operators in the Heisenberg algebra. 

Compactly supported smooth functions $C_c^\infty(T_HX)$ form a two-sided ideal in the convolution algebra $\EE'(T_HX)$ of compactly supported distributions on the groupoid $T_HX$.
The asymptotic expansion for products of full symbols in the Heisenberg calculus implies that if the principal symbol of an operator $P$ is invertible, then the {\em full} symbol $\sigma_H(P)$ has an inverse modulo the ideal $C_c^\infty(T_HX)$.
By a  general argument explained in \cite[II.9.$\alpha$]{Co94}, 
this implies that $\sigma_H(P)$ has an `index' in the $K$-theory group $K_0(C^*(T_HX))$.
First, the full symbol $\sigma_H(P)$ determines an element in {\em algebraic} relative $K$-theory,
\[ [\sigma_H(P)]\in K_0(\EE'(T_HX), C_c^\infty(T_HX)).\]
By excision we have
\[ K_0(\EE'(T_HX), C_c^\infty(T_HX))\cong K_0(C_c^\infty(T_HX)),\]
and, finally, the inclusion of rings $C_c^\infty(T_HX)\subset C^*(T_HX)$ maps the element $[\sigma_H(P)]$ to the $K$-theory group $K_0(C^*(T_HX))$ (which is the same in algebraic and $C^*$-algebraic $K$-theory).
To make this more explicit, if $Q$ is a parametrix for $P$ (an inverse modulo smoothing operators), then the $K$-theory element
\[ [\sigma_H(P)]=[e]\ominus [e_0]\in K_0(C^*(T_HX))\]
is represented by the formal difference of the idempotents
\[e=\left(\begin{array}{cc}S_0^2&S_0\sigma_H(Q)\\S_1(1+S_1)\sigma_H(P)&1-S_1^2\end{array}\right),\;
e_0=\left(\begin{array}{cc}0&0\\0&1\end{array}\right),\]
where $S_0=1-\sigma_H(Q)\sigma_H(P)$, $S_1=1-\sigma_H(P)\sigma_H(Q)$.
(See \cite[II.9.$\alpha$]{Co94} for the details of the general construction.)

\subsection{The ``Choose an operator'' maps}
 
In this section we prove the analog of the Poincar\'e duality Theorem \ref{Poincare} for analytic $K$-homology $KK(C(X),\CC)$.
The proof of this theorem closely follows the steps of \cite[section 3.7]{vE10a},
generalizing everything from a statement in $K$-theory to a statement about the functor $KK(C(X),-)$.
We refer the reader to \cite{vE10a} for details of the proof in $K$-theory, and indicate here how the argument can be stengthened to prove Theorem \ref{choose}.
\vskip 6pt 

The Poincar\'e duality map
\[ {\rm Op}_e\;\colon\; K^0(T^*X)\to KK^0(C(X),\CC)\]
is defined as follows.
Given $\sigma\in K^0(T^*X)$, choose an elliptic zero order pseudodifferential operator $P$ on $X$ such that in $K^0(T^*X)$,
\[ [\sigma(P)]=\sigma.\]
Then ${\rm Op}_e(\sigma)$ is the element in $KK^0(C(X),\CC)$ determined by $P$.
The non-trivial point here is that if two elliptic pseudodifferential operators $P, Q$ have $[\sigma(P)]=[\sigma(Q)]$ in $K^0(T^*X)$, then $P, Q$ determine the same element in $KK^0(C(X),\CC)$---i.e., a homotopy of symbols can be lifted to a homotopy of operators.
(See \cite{AS1}, \cite[section 23]{BD80}.)
Instead of attempting a direct construction for the Heisenberg
pseudodifferential calculus,
we will, in the proof below, construct a map
\[ {\rm Op}_H\;\colon\; K_0(C^*(T_HX))\to KK(C(X),\CC),\]
and then prove a posteriori that if $P$ is an order zero operator in the Heisenberg calculus with principal symbol $\sigma_H(P)$,
then
\[ {\rm Op}_H([\sigma_H(P)]) = [P].\]
This justifies referring to the map ${\rm Op}_H$ as the ``choose a hypoelliptic operator'' map for the Heisenberg calculus.

\begin{theorem}\label{choose}
There exists a map 
\[ {\rm Op}_H\;\colon\; K_0(C^*(T_HX))\to KK^0(C(X),\CC),\]
such that  $(i)$ if $P$ is an order zero operator in the Heisenberg algebra with invertible principal symbol $\sigma_H(P)$ then ${\rm Op}_H(\sigma_H(P))=[P]$ in $KK^0(C(X), \CC)$,
and $(ii)$ the following diagram commutes,
\[ \xymatrix{    K^0(T^*X) \ar[r]^-{\Psi}\ar[d]^{\rm Op_{e}} 
               & K_0(C^*(T_HX)) \ar[ld]^{\rm Op_{H}}  \\
                 KK^0(C(X), \CC)
               &  }
\]
\end{theorem}
\begin{proof}
The proof is a strengthening of the groupoid arguments developed in \cite{vE10a}.
The main difference is that we need to replace $K$-theory $KK(\CC,-)$ by the functor $KK(C(X),-)$.
What makes this possible is that all $C^*$-algebras that play a role in our argument are of Type I, and hence nuclear.
Therefore all ideals are semi-split (by \cite{CE76}), 
and since $C(X)$ is separable ($X$ is a compact manifold) the $KK(C(X),-)$ functor has $6$-term exact sequences \cite[Thm. 19.5.7]{Bl98}.
As a result, the quotients by various contractible ideals that play a role in the tangent groupoid argument induce isomorphisms not only in $K_0(-)\cong KK(\CC, -)$, but also in $KK(C(X),-)$.

We briefly review the argument and indicate how it can be strengthened.
For a differential Heisenberg operator $P$
with invertible principal symbol 
we constructed in \cite[section 3.6]{vE10a}  an element in the $K$-theory of the parabolic tangent groupoid,
\[ [\mathbb{P}] \in K_0(C^*(\THX)).\]
The `parabolic tangent groupoid' $\THX$ is algebraically the disjoint union of $T_HX$ and the family of pair groupoids $X\times X\times (0,1]$, similar to Connes' tangent groupoid.
The `blow-up' of the diagonal in $X\times X$ is defined using the parabolic dilations $\delta_s$ in the Heisenberg groups $G_x$ \cite[section 3]{vE10a}.
The class $[\mathbb{P}]$ combines the principal Heisenberg symbol of $P$ with $P$ itself into a single $K$-theory element for the tangent groupoid.

We must modify this construction and apply it to order zero operators.
This is actually a considerable simplification,
especially if we construct $[\mathbb{P}]$ in $KK$-theory.
Let $\EE=C^*(\THX)\oplus C^*(\THX)$ be the obvious $\ZZ_2$-graded Hilbert module over $C^*(\THX)$.
The adjointable operator 
\[ \mathbb{P}\in \LL(\EE)\cong M(C^*(\THX))\otimes M_2(\CC)\]
restricts to the principal symbol 
\[ \left(\begin{array}{cc}0&P_x^*\\P_x&0\end{array}\right)\;\in M(C^*(G_x))\otimes M_2(\CC) \]
for the fiber at $s=0$ and $x\in X$,
and to
\[ \left(\begin{array}{cc}0&P^*\\P&0\end{array}\right)\;
\in \LL(L^2(M)\oplus L^2(M)) \cong M(C^*(X\times X))\otimes M_2(\CC)\]
for all other values $s>0$.
The principal symbol is, by definition, the highest order part in an asymptotic expansion of the operator kernel of $P$, and these asymptotics agree precisely with the parabolic `blow-up' of the diagonal in $X\times X$ that defines $\THX$.

What we did not realize in \cite{vE10a} is that $(\EE,\mathbb{P})$ actually defines an element in the $KK$-group
\[ [\mathbb{P}] \in KK(C(X), C^*(\THX)).\]
There is an obvious diagonal representation 
\[ \phi\;\colon\; C(X)\to \LL(\EE)=M(C^*(\THX))\otimes M_2(\CC).\]
At $s=0$ it restricts to the identification of $C(X)$ with the center of $M(C^*(T_HX))$,
while at $s>0$ we have the representation of continuous functions on $X$ as multiplication operators in $M(C^*(X\times X))\cong \LL(L^2(X))$.
One easily verifies compactness of commutators $[\mathbb{P}, \phi(f)]\in \KK(\EE)$
for continuous functions $f\in C(X)$. 

Just as in \cite{vE10a}, the restriction map at $s=0$ has a contractible ideal and therefore induces an isomorphism
\[ e_0\;\colon\; KK(C(X), C^*(\THX))\cong KK(C(X), C^*(T_HX)).\]
Observe that, by construction of the Fredholm module $[\mathbb{P}]$,
we have $e_0([\mathbb{P}])=[\sigma_H(P)]$.

Restriction to the fiber at $s=1$ induces a map
\[ e_1\;\colon\; KK(C(X), C^*(\THX))\to KK(C(X), C^*(X\times X))\cong KK(C(X), \CC).\]
Again, by construction, $e_1([\mathbb{P}])=[P]$.
It follows that the combined map
\[ e_1\circ e_0^{-1}\;\colon\;  KK(C(X), C^*(T_HX))\to KK(C(X),\CC),\]
is the ``choose a hypoelliptic operator'' map,
\[ e_1\circ e_0^{-1}([\sigma_H(P)]) = [P].\]
The exact same argument, using the tangent groupoid $\TX=TX\sqcup X\times X\times (0,1]$ of Connes,
works for elliptic order zero operators.

Having proven that the ``choose an operator'' map is induced by the appropriate tangent groupoid,
we can now show that these two maps commute with the isomorphism $\Psi$.
The argument from \cite[section 3.8]{vE10a} involving the {\em adiabatic groupoid} of the parabolic tangent groupoid $\THX$ (a deformation of a deformation)
applies, without essential change, to show that the diagram
\[ \xymatrix{    KK(C(X), C_0(T^*X)) \ar[r]^-{\Psi}\ar[d] 
               & KK(C(X), C^*(T_HX)) \ar[ld] \\
                 KK((C(X), \CC)
               &  }
\]
commutes. 

If $\HH_A$ denotes the standard countably generated free Hilbert $A$-module,
then any $\ast$-homomorphism $\phi\,\colon C(X)\to Z(M(A))$ 
induces, in the obvious way, a map $\tilde{\phi}\,\colon C(X)\to Z(\LL(\HH_A))$.
Thus, for a $C(X)$-$C^*$-algebra $A$ there is a natural homomorphism
\[ \alpha_X\;\colon\;KK(\CC, A) \to KK(C(X),A).\]
The map $\alpha_X$ is a natural transformation from the $KK(\CC,-)$ functor to the $KK(C(X),-)$ functor on the category of $C(X)$-$C^*$-algebras.
In fact, it is a natural transformation for the $KK^X$-category.

Therefore, since the isomorphism $\Psi$ is implemented by a $KK^X$-element
\[ \Psi\in KK^X(C_0(T^*X), C^*(T_HX))\]
we have a commutative diagram
\[ \xymatrix{    KK(\CC, C_0(T^*X)) \ar[r]^-{\Psi}\ar[d]_{\alpha_X} 
               & KK(\CC, C^*(T_HX)) \ar[d]^{\alpha_X} \\
                 KK(C(X), C_0(T^*X)) \ar[r]^-{\Psi} 
               & KK(C(X), C^*(T_HX)) }
\]
Composing this with our result above yields the proposition.

\end{proof}

\subsection{The hypoelliptic index theorem in $K$-homology}

The Atiyah-Singer theorem for elliptic operators amounts to commutativity of the triangle \cite[Thm. 23.1]{BD80},
\[ \xymatrix{  & K^0(T^*X) \ar[dl]_{c}\ar[dr]^{{\rm Op_e}} \\
                 K_0(X) \ar[rr]^-{\mu}  
              && KK_0(C(X),\CC).
            }
\]
We obtain an analogous theorem for the Heisenberg calculus.

\begin{theorem}\label{triangle}
Let $(X, H)$ be a closed contact manifold.
Then there is commutativity in the diagram
\[ \xymatrix{  & K_0(C^*(T_HX)) \ar[dl]_{b}\ar[dr]^{{\rm Op_H}} \\
                 K_0(X) \ar[rr]^-{\mu}  
              && KK_0(C(X),\CC)
            }
\]
Equivalently, the geometric $K$-cycle that corresponds to the analytic $K$-cycle $[P]$ determined by a Heisenberg-elliptic operator is
\[\mu^{-1}([P]) = b(\sigma_H(P))\] 
\end{theorem}
\begin{proof}
The theorem is implied by commutativity of the three smaller triangles in the  diagram,
\[ \xymatrix{   & K_0(C^*(T_HX)) \ar@/^/[ddr]^{\rm Op_H} \ar@/_/[ddl]_b & \\
               & K^0(T^*X) \ar[u]^{\Psi}\ar[dr]_{\rm Op_e}\ar[dl]^c& \\
 K_0^{top}(X) \ar[rr]_\mu & & KK^0(C(X),\CC)              
 }
\]
\end{proof}
Modulo an explicit computation of the Poincar\'e dual $b(\sigma_H(P))$ of the Heisenberg symbol of a hypoelliptic operator, Theorem \ref{triangle} solves the general index problem for hypoelliptic operators in the Heisenberg calculus.
In section \ref{computing_b} we discuss the explicit computation of the $(M,E,\varphi)$ cycle that corresponds to a Heisenberg-elliptic operator, and  apply it in concrete examples.

\section{Computation of the $K$-cycle}\label{computing_b}

In this section we explicitly calculate $\mu^{-1}([P]) = b(\sigma_H(P))$.
A number of examples are then considered.

\subsection{Perturbing the symbol}
Consider a Heisenberg-elliptic operator
\[ P\;\colon\; C^\infty(X,F^0)\to C^\infty(X,F^1)\]
acting on sections in a $\CC$ vector bundle $F^0$,  whose range consists of sections in a vector bundle $F^1$.
We wish to compute the image $b(\sigma_H(P))$ of the $K$-theory element $[\sigma_H(P)]$ associated to the Heisenberg symbol of a Heisenberg-elliptic operator $P$ under the noncommutative Poincar\'e duality map,
\[ b\;\colon\; K_0(C^*(T_HX))\to K_0(X)\]
Following Definition \ref{map_b}, the first step is to lift the element $[\sigma_H(P)]$ in the $K$-theory group $K_0(C^*(T_HX))$ to an element in $K_0(I_H)$.

The full Heisenberg symbol of $P$  is a smooth family $\sigma_H(P)=\{P_x, x\in X\}$ of operators
\[ P_x\;\colon\; C_c^\infty(G_x, F^0_x)\to C_c^\infty(G_x, F^1_x)\]
Each $P_x$ is an operator of convolution with a compactly supported distribution on $G_x$, where $G_x=H_x\times \RR$ is the Heisenberg group that is the fiber at $x\in X$ of $T_HX$.
Taken together, the operators $\{P_x, x\in X\}$ correspond to a compactly supported distributional section  $\sigma_H(P)$ of the bundle ${\rm Hom}(\pi^*F^0,\pi^*F^1)$ on $T_HX$.
Recall that the operator $P$ determines the distribution $\sigma_H(P)$ on $T_HX$ (its `full symbol') up to a perturbation by a compactly supported smooth section defined on $T_HX$.

Let $\pi_0$ denote the representation of the Heisenberg group that assembles all scalar unitary  representations.
At the level of the group algebra, $\pi_0$ is composition of the Fourier transform $C_c^\infty(G)\to C^\infty(V^*\times \RR^*)$
with restriction to $s=0$ in $\RR^*$,
\[ C_c^\infty(G)\to C^\infty(V^*)\]
For scalar distributions on $T_HX$ this representation induces the algebra homomorphism
\[ \pi_0\;\colon \; \EE'(T_HX)\to C^\infty(H^*)\]
where $\EE'(T_HX)$ is a convolution algebra, while $C^\infty(H^*)$ is the algebra with pointwise multiplication of functions.
 
If we take the vector bundles $F^0, F^1$ into account, then the scalar representations of $P_x$ assemble to a function
\[ \pi_0(P_x) \;\colon \;H_x^*\to \mathrm{Hom}(F_x^0,F_x^1)\]
and what we get is the classical symbol of $P$ restricted to $H^*$, 
\[ \pi_0(\sigma_H(P)) \in C^\infty(H^*, \mathrm{Hom}(\pi^*F^0,\pi^*F^1))\]
where $\pi^*F^j$ is the pull-back of $F^j$ from $X$ to $H^*$.
This classical part of the Heisenberg symbol is `elliptic' in the sense that it is invertible outside a compact set in $H^*$.

As a class in $K^0(H^*)$ the partial symbol $\pi_0(\sigma_H(P))$ is trivial.
To see this, consider the exact sequence in $K$-theory
\[ K^0(T^*X\setminus H^*) \to K^0(T^*X)\to K^0(H^*)\]
The first map is surjective, and so the second map is trivial.
Therefore the same is true for the isomorphic sequence  
\[ K_0(I_H) \to K_0(C^*(T_HX))\stackrel{\pi_0}{\longrightarrow} K_0(C^*(H))\]
Since  $\pi_0(\sigma_H(P))$ is $K$-theoretically trivial we may assume that,
after a compactly supported perturbation of $\pi_0(\sigma_H(P))$, we have
\[ \pi_0(\sigma_H(P)) \in C^\infty(H^*, \mathrm{Iso}(\pi^*F^0,\pi^*F^1)).\]
(We may have to stabilize $F^0, F^1$.)
Note that we can extend such a perturbation to the full Heisenberg symbol $\sigma_H(P)$. 
We shall assume, from now on, that $\pi_0(\sigma_H(P))$ is invertible on all of $H^*$.

We make one further modification to the symbol.
If we restrict the (perturbed) $\pi_0(\sigma_H(P))$ to the zero section in $H^*$,
the resulting section in ${\rm Hom}(F^0,F^1)$ is an isomorphism of vector bundles
\[ \sigma_0\;\colon\; F^0\to F^1\]
We compose the operator $P$ with the vector bundle isomorphism $\sigma_0^{-1}$, andobtain a new operator 
\[ \sigma_0^{-1}\circ P\;\colon\; C^\infty(X,F^0)\to C^\infty(X,F^0)\]
that represents, of course, the same $K$-homology class as $P$.
The point of this modification is that the equatorial symbol of $\sigma_0^{-1}\circ P$ (i.e., its value at the spherical boundary $S(H^*)$ of $H^*$) is homotopic to the constant map from $S(H^*)$ to the identity operator in the fibers of $F^0$.
In fact, we can choose the isomorphism $\sigma_0\colon F^0\to F^1$ to be {\em any} isomorphism with this property. 


With these modifations, the full Heisenberg symbol $\sigma_0^{-1}\sigma_H(P)$ defines an element in the convolution algebra $\EE'(T_HX, {\rm End}(\pi^* F^0))$, and it is invertible modulo the ideal
\[  I^\infty_H := \{ x\in C_c^\infty(T_HX, {\rm End}(\pi^* F^0)) \,\mid\, \pi_0(x) = 0\}\]
Since the closure of $I_H^\infty$ is $I_H$, the formal difference of idempotents associated to $\sigma_0^{-1}\sigma_H(P)$, as defined in section \ref{K-symbol}, is the desired element in $K_0(I_H)$.

\subsection{Suspension in $K$-theory}

The next step in Definition \ref{map_b} is the isomorphism
\[ K_0(I_H)\cong K^1(X)\oplus K^1(X)\]
which is a composition of Morita equivalence and suspension. 
We may reverse the order of the Morita equivalence and the suspension isomorphisms.
Let $B$ be the $C^*$-algebra
\[ B=\KK(V_-^{BF})\oplus \KK(V^{BF}_+)\]
In section \ref{NCT} we defined an explicit isomorphism
\[ I_H\cong C_0(\RR)\otimes B\]
We will first prove a general lemma in $K$-theory that allows us to compute the suspension isomorphism
\[ K_0(C_0(\RR)\otimes B) \cong K_1(B)\]
We subsequently consider the effect of the Morita equivalence $B\sim C(X)\oplus C(X)$. 
 
\begin{lemma}\label{K-one}
Let $B$ be a separable $C^*$-algebra.
Let $u\in M_n(B)^+$ be a unitary matrix that represents a class $[u]\in K_1(B)$,
and let $\theta$  denote the suspension isomorphism,  
\[ \theta\;\colon\; K_1(B)\to K_0(C_0(0,1)\otimes B).\]
If $F(t)$, $t\in[0,1]$ is the norm continuous family that linearly interpolates between $F(0)=1$ and $F(1)=u$,
\[ F(t)=1-t+tu \in M_n(B)^+,\] 
then $F\in C_0([0,1], M_n(B))^+$ is invertible modulo $C_0((0,1), M_n(B))$,
and defines a relative class
\[ [F]\in K_0(C_0(0,1)\otimes B).\]
Then $[F]=\theta([u])$.
\end{lemma}
\begin{proof}
We will prove this by making use of the flexibility of $KK$-theory.
For ease of notation, let us write $A=C_0((0,1), B)$.
Since $M_n(B)^+\subset M_n(M(A))=\LL(A^n)$, we can think of $F$ as an adjointable operator on the Hilbert $A$-module $A^n$.
The fact that $F_0=1$, $F_1=u$ and $F_t=1$ modulo $M_n(B)$ for all $t\in [0,1]$
implies that $F$ is invertible modulo compact operators $\KK(A^n)=C_0((0,1),M_n(B))$.
In other words, the relative $K$-theory class defined by $F$ can also be represented as a Kasparov module
\[ [F\,\colon A^n\to A^n]\in KK(\CC, A).\]
We will prove that this Fredholm module corresponds to $\theta([u])$ under the standard identification
$KK(\CC, A)\cong K_0(A)$.

Choose a continous family of unitaries $z(t)\in M_{2n}(B^+), t\in[0,1]$ such that
\[ z_0=\left(\begin{array}{cc} 1_n&0\\0&1_n\end{array}\right),\;
 z_1=\left(\begin{array}{cc} u&0\\0&u^{-1}\end{array}\right).\]
Let 
\[ p_n=\left(\begin{array}{cc} 1_n&0\\0&0\end{array}\right),\]
and let $e$ denote the continuous family of projections
\[ e_t = z_t p_n z_t^{-1}.\]
Observe that $e_0=e_1=p_n$.
Then the proof of [Blackadar, Theorem 8.2.2] implies that
\[ \theta([u]) = [e]-[p_n]\in K_0(C_0((0,1),B)).\]
Consider the family of partial isometries 
\[ G(t) = z(t)p_n.\]
Observe that $G(0) = p_n$, $G(1) = u$ and that $G(1) = p_n$ modulo $\KK(A^{2n})$, while
\[ 1-G^*G=1-p_n,\; 1-GG^*=1-e.\]
Let $\phi\,\colon\CC\to \LL(B^{2n})$ denote the homomorphism with
$\phi(1)=p_n$.
With this choice of a nonunital map $\phi$ the operator $G$ satisfies the Fredholm property, and we obtain a Kasparov module
\[ [G\,\colon A^{2n}\to A^{2n}, \phi]\in KK(\CC, A).\]
Moreover, the module $[G\,\colon A^{2n}\to A^{2n}, \phi]$ is homotopic to $[F\,\colon A^n\to A^n, p_n\phi p_n]$, where $p_n\phi p_n\,\colon\CC\to \LL(A^n)$ is just the standard unital map.
But unlike $F$ the operator $G$ has closed range, and
we can take its index
\[ {\rm Index}\,G = [1-G*G]-[1-GG^*]=[e]-[p_n]=\theta([u]).\]
It follows that $[G]$ corresponds to $\theta([u])$,
and therefore so does $[F]$.
\end{proof}

\subsection{The general formula}

We now derive an explicit formula for the $K$-homology element $\mu^{-1}([P]) = b(\sigma_H(P))$ for an arbitrary Heisenberg-elliptic operator $P$.

\begin{theorem}\label{general_formula}
Let $P$ be a Heisenberg-elliptic operator on a closed oriented contact manifold that acts on sections in a complex vector bundle,
\[ P\;\colon\; C^\infty(X,F^0)\to C^\infty(X, F^1)\]
Let $\sigma_0\colon F_0\to F_1$ be a vector bundle isomorphism
such that the pullback of $\sigma_0$ to the sphere bundle $S(H^*)$
is homotopic to the equatorial symbol $\sigma(P)|S(H^*)$.

Then the element in $K$-homology determined by the Fredholm operator $P$ is
\[[P] = \pi_+(\sigma_0^{-1}\sigma^+_H(P))\cap [X^+] + \pi_-(\sigma_0^{-1}\sigma^-_H(P))\cap [X^-]\]
Here $\pi_\pm$ denote the Bargmann-Fok representations on the Hilbert modules $V^{BF}_\pm\otimes F^0$ of the two components of the principal Heisenberg symbol of the operator $\sigma_0^{-1}\circ P$.
\end{theorem}

\begin{proof}
Choose hermitian structures for the vector bundles $F^0, F^1$,
such that $\sigma_0\colon F^0\to F^1$ is a unitary isomorphism.
For technical reasons, replace $P$ with  
\[ \tilde{P}=\sigma_0^{-1}\circ P(1+P^*P)^{-1/2} \]
Then $\tilde{P}$ is an order zero Heisenberg pseudodifferential operator that represents  the same $K$-homology class as $P$. But the {\em principal} Heisenberg symbol of $\tilde{P}$ is a unitary.

Let $B=\KK(V_-^{BF})\oplus \KK(V^{BF}_+)$ and
\[ I_H \cong C_0(\RR,B)\cong C_0((0,1), B)\] 
where we choose an arbitrary orientation preserving homeomorphism $\RR\approx (0,1)$.

Recall that the equatorial symbol of $\tilde{P}$ is homotopic to the identity.
We first assume, for simplicity, that the equatorial symbol of $\tilde{P}$ is {\em equal} to the identity.
Then we can define the unitary 
\[ u=\pi_-(\sigma^-_H(\tilde{P})) \oplus \pi_+(\sigma^+_H(\tilde{P}))\in B^+\] 
The difference between the full symbol $\sigma_H(\tilde{P})$ and $1-t+tu$ is an element in $C_0((0,1),B)$,
so that $\sigma_H(\tilde{P})$ and $1-t+tu$ represent the same element in $K_0(I_H)$.
Therefore, by Lemma \ref{K-one}, the $K$-theory element
\[ [u]\in K_1(B) \cong K_1(\KK(V_-^{BF})) \oplus K_1(\KK(V_+^{BF}))\]
corresponds, under the suspension isomorphism, to $sigma_H(\tilde{P})]\in K_0(I_H)$.
Completing the steps in Definition \ref{map_b}, we obtain the formula stated in the theorem.

In general, the equatorial symbol of $\tilde{P}$ is {\em homotopic} to the constant map from $S(H^*)$ to the identity operator on $F^0$, and so essentially the same argument applies.
\end{proof}

\begin{remark}
In Theorem \ref{general_formula} we did not explictly address the Morita equivalence that is part of the definition of the map $b$.
Strictly speaking, we only defined elements
\[ [\pi_\pm(\sigma_0^{-1}\sigma^\pm_H(P))] \in K_1(\KK(\HH^{BF}_\pm))\]
but we have not yet indicated what to do about the Morita equivalences
\[ K_1(\KK(\HH^{BF}_\pm))\cong K^1(X)\]
But this is a standard procedure in $K$-theory.
Let $V^N_\pm \to X$ denote the complex vector bundles on $X$ defined by
\[  V^N_+ = \bigoplus_{j=0}^N\;{\rm Sym}^j\,H^{1,0},\;\;V^N_- = \bigoplus_{j=0}^N\;{\rm Sym}^j\,H^{0,1},\]
and let
\[ S^N_\pm\;\colon\; V^{BF}_\pm \to V^N_\pm\]
be the family of orthogonal projections in each  fiber.
Let us denote by $\pi_\pm^N$ the Bargmann-Fok representations
compressed by the projections $S^N$,
\[ \pi^N_\pm(a) := S^N_\pm \pi_\pm(a) S^N_\pm\]
For sufficiently large value of the integer $N$ we have
\[ [\pi_\pm(\sigma_0^{-1}\sigma^\pm_H(P))] = [S^N\pi_\pm(\sigma_0^{-1}\sigma^\pm_H(P))S^N+ (1-S^N)]\in K_1(\KK(\HH^{BF}_\pm)),\]
where the correct value of $N$ depends on $\pi_\pm(\sigma_0^{-1}\sigma^\pm_H(P))$.
Observe that with this choice of $N$,
the elements $\pi_\pm^N(\sigma_0^{-1}\sigma^\pm_H(P))$ define automorphisms of the vector bundles $V^N_\pm$, so that
\[ [V^N, \pi^N_\pm(\sigma_0^{-1}\sigma^\pm_H(P))]\in K^1(X)\]
The compression
\[ K_1(\KK(\HH^{BF}_\pm))\to K^1(X)\;\colon\; [\pi_\pm(\sigma_0^{-1}\sigma^\pm_H(P))]\mapsto [V^N_\pm, \pi^N_\pm(\sigma_0^{-1}\sigma^\pm_H(P))]\]
implements the Morita equivalence (where the size of $N$ depends on $P$, as mentioned).
So we have, more precisely,
\[ [P]= [\pi^N_-(\sigma_0^{-1}\sigma^-_H(P))]\,\cap\,[X^-] \;+\; [\pi^N_+(\sigma_0^{-1}\sigma^+_H(P))]\,\cap\, [X^+] .\]
\end{remark}

\subsection{Toeplitz operators}\label{Toeplitz}

The index formula for Toeplitz operators of Louis Boutet de Monvel \cite{Bo79} is a special case  of Theorem \ref{general_formula}.
If the contact manifold $X$ is the boundary of a strictly pseudoconvex complex domain,
then the {\em Szeg\"o  projector} $S$ is defined as the projection of $L^2(X)$ onto the Hardy space $H^2(X)$.
Let $f$ be a smooth map  
\[ f\,\colon X\to {\rm GL}(r, \CC).\]
In an evident fashion, the function $f$ defines a multiplication operator $M_f$ on $L^2(X)^{\oplus r}$. 
The Toeplitz operator $T_f$ is the compression of  $M_f$ to $H^2(X)^{\oplus r}$,
i.e.,
\[ T_f = S_rM_f S_r,\]
where $S_r=S^{\oplus r}$.
The Toeplitz operator $T_f$ is a Fredholm operator, and Boutet de Monvel's formula is
\[ {\rm Index}\,T_f = \ang{{\rm ch}(f)\cup {\rm Td}(X), [X]},\]
where ${\rm Td}(X)$ is the Todd class of the Spin$^c$ manifold $X^+$,
and $[f]$ is viewed as an element in $K^1(X)$.

As shown in \cite{Ep04}, the operator
\[ \tilde{T}_f = T_f + 1-S_r\]
is an order zero operator in the Heisenberg calculus.
It is immediate that $\tilde{T}_f$ is a Fredholm operator  on $L^2(X)^{\oplus r}$ with ${\rm Index}\,\tilde{T}_f={\rm Index}\,T_f$.
To derive Boutet de Monvel's result from our general formula,
we use the calculation in \cite{Ep04} of the principal Heisenberg symbol of $\tilde{T}_f$.
Within the Bargmann-Fok space $\HH^{BF}_+$ there is the vacuum summand $\CC={\rm Sym}^0 \CC^n$. 
Hence, the Hilbert module $V_+^{BF}$ contains the trivial line bundle $\underline{\CC}={\rm Sym}^0H^{1,0}$.
The Bargmann-Fok representations $\pi_\pm$ of the Heisenberg symbol of $\tilde{T}_f$ are
\begin{align*}
\pi_+(\sigma_H(\tilde{T}_f)) &= f \;\text{\rm acting on}\; \underline{\CC}^{r},\\ 
\pi_+(\sigma_H(\tilde{T}_f)) &= 1 \;\text{\rm on the orthogonal complement of }\; \underline{\CC}^{r},\\ 
\pi_-(\sigma_H(\tilde{T}_f))&=1
\end{align*} 
The equatorial symbol of $\tilde{T}_f$ is $1$, and the formula of Theorem \ref{general_formula} 
\[ [T_f] = [\pi_+(\sigma^+_H(\tilde{T}_f))] \cap [X^+]
\,+\,[\pi_-(\sigma^-_H(\tilde{T}_f))] \cap [X^-] \]
reduces to
\[ [T_f] = [f]\cap [X^+]\]
Equivalently, the $K$-cycle which solves the index problem for $\tilde{T}_f$ is $(X^+\times S^1, E_f, \varphi)$ where $E_f$ is the vector bundle on $X\times S^1$ obtained from a clutching construction---using trivial vector bundles of fiber dimension $r$---with $f$, and $\varphi\,\colon X\times S^1\to X$ is the projection.  
In particular, the formula of Boutet de Monvel is now a corollary.

\subsection{Second order scalar operators}\label{final}

For Heisenberg-elliptic {\em differential} operators $P$, 
the principal Heisenberg symbol and its action on the Bargmann-Fok spaces 
can be explicitly computed by an algorithmic procedure.
Theorem \ref{general_formula} then gives an explicit $(M,E,\varphi)$-cycle that corresponds to  $[P]$.
We will illustrate this by computing the geometric $K$-cycle for the  Heisenberg-elliptic operators $P_\gamma = \Delta_H+i\gamma T$.
The procedure is essentially the same for all {\em differential} Heisenberg-elliptic operators,
but it is most easily illustrated by an explicit example.
\vskip 6pt
Because of Darboux's theorem the contact manifold $X$ can locally be identified with an open subset of the Heisenberg group $\RR^{2n+1}$.
Let $X_j, Y_j, T$ be the standard right invariant vector fields on the Heisenberg group with $[X_j, Y_j]=T$.
In such local coordinates, the operator $P_\gamma$ can be written as
\[ P_\gamma = \sum_{j=1}^n -(X_j^2+Y_j^2)+i\gamma T+\dots,\]
where we ignore first order terms in $X_j, Y_j$.
These lower order terms may appear because the sublaplacian $\Delta_H=\sum -(X_j^2+Y_j^2)+\dots$ is unique only up to lower order terms.

We have $JX_j=Y_j$, and so $Z_j=(X_j-iY_j)/\sqrt{2}$.
A simple computation shows that $2Z_j\bar{Z}_j=X_j^2+Y_j^2+i[X_j,Y_j]$, and so
\[ \sum -2Z_j\bar{Z}_j = \sum -(X_j^2+Y_j^2) +inT.\]
We can therefore write
\[ P_\gamma = \sum -2Z_j\bar{Z}_j -i(n-\gamma)T+\dots.\]
The Bargmann-Fok representation on $(\HH^{BF}_+)_x$ of the symbol $\sigma_H(P)=\{P_x,x\in X\}$ is
\[ \pi_+(P_x) = \sum 2z_j\frac{\partial}{\partial z_j} +n-\gamma(x).\]
The symmetric powers of $H^{1,0}$ are eigenspaces of $\pi_+(P_x)$.
To see this, observe that 
\[ \pi_+(P_x)z^\alpha = (2|\alpha|+n-\gamma(x))z^\alpha,\]
where $z^\alpha=z_1^{\alpha_1}\dots z_n^{\alpha_n}$ and $|\alpha|=\alpha_1+\dots+\alpha_n$.
Therefore on ${\rm Sym^j}H^{1,0}$ the operator $\pi_+(P_x)$ acts as the scalar 
\[ a_j(x) = 2j+n-\gamma(x)\]
In the dual Bargmann-Fok representation the roles of $Z$, $\bar{Z}$ are reversed.
So from
\[ \sum -2\bar{Z}_jZ_j = \sum -(X_j^2+Y_j^2) -inT,\]
we get
\[ P_\gamma = \sum -2\bar{Z}_jZ_j +i(n+\gamma)T.\]
The representation of $\sigma_H(P)$ on the conjugate Bargmann-Fok spaces $(\HH^{BF}_-)_x$ is then given by
\[ \pi_-(P_x) = \sum 2\bar{z}_j\frac{\partial}{\partial \bar{z}_j} +n+\gamma(x).\]
and we find that $\pi_-(P_x)$ acts on ${\rm Sym^j}H^{0,1}$
as the scalar 
\[ b_j= 2j+n+\gamma(x)\]
Replacing $P$ with 
\[\tilde{P} = P(1+P^*P)^{-1/2}\]
we obtain an order zero operator whose principal symbols
act on the symmetric powers of $H^{1,0}$ by the scalar $a_j/\sqrt{1+|a_j|^2}$,
and similarly for $H^{0,1}$.
For large values of $j$ these renormalized scalars are close to $1$,
and so they contribute only trivial summands to the $K$-cycle
and can be ignored.
We can use the non-normalized scalars, because as vector bundle automorphisms they are homotopic to their normalized versions.
The result is the following formula in $K$-homology.

\begin{proposition}
Let $P_\gamma = \Delta_\gamma + i\gamma T$ be a Heisenberg-elliptic operator.
Then the element in $K$-homology represented by $P_\gamma$ is
\[ \sum_{j} \left( [{\rm Sym}^j H^{1,0}] \cup [2j+n-\gamma]\right)\cap [X^+]\,+\, \sum_{j} \left([{\rm Sym}^j H^{0,1}]\cup [2j+n+\gamma]\right)\cap[X^-]\]
Here the vector bundles ${\rm Sym}^j H^{1,0}$ and ${\rm Sym}^j H^{0,1}$ represent elements in the $K$-theory group $K^0(X)$, while the functions $2j+n-\gamma$ and $2j+n+\gamma$ represent elements in  odd $K$-theory $K^1(X)$.
\end{proposition}

\begin{corollary}\label{Chern-char}
The Chern character of the $K$-cycle $[P_\gamma]$ is the Poincar\'e dual of the cohomology class
\begin{align*}
{\rm P.D.}\, {\rm ch}\,([P_\gamma])=&\sum_{j=0}^N \;{\rm ch}(2j+n-\gamma)\cup {\rm ch}({\rm Sym^j}\, H^{1,0})\cup e^{c_1(H^{1,0})/2}\cup \hat{A}(X)\\
&+ (-1)^{n+1}\sum_{j=0}^N\;{\rm ch}(2j+n+\gamma)\cup {\rm ch}({\rm Sym^j}\, H^{0,1})\cup e^{c_1(H^{0,1})/2}\cup \hat{A}(X).
\end{align*}
Here the odd Chern character ${\rm ch}(f)\in H^1(X,\ZZ)=[X,S^1]$
refers to the $1$-cocycle associated to a continuous map $f\,\colon X\to \CC^\times \sim S^1$.
\end{corollary}

\begin{example}
In \cite{vE10c} we derived an explicit topological formula for the Fredholm index of $P_\gamma$ in the simplest possible case where $X$ is a three manifold.
We now see that this index formula, while correct, was incomplete.
Let us calculate the $K$-cycle, and compare it with the index formula in \cite{vE10c}.

On a three manifold $X$ we have $\hat{A}(X)=1$, while the bundle $H^{1,0}$ is a line bundle. Then ${\rm Sym^j}H^{1,0}=(H^{1,0})^{\otimes j}$.
Writing $c_1=c_1(H^{1,0})$ we have
\[ {\rm ch}({\rm Sym^j} H^{1,0})e^{c_1/2}
= (1+j c_1)(1+\frac{1}{2} c_1)
= 1+\frac{2j+1}{2}c_1.\]
Denote by 
\[ W_k ={\rm ch}(\gamma - k)=\left[-\frac{1}{2\pi i}\frac{d\gamma}{\gamma - k}\right] \in H^1(X,\ZZ)\]
the $1$-cocycle that encodes the winding of the coefficient $\gamma$ around the odd integer $k$,
\[ \gamma\;\colon\; X\to \CC\setminus \{\text{\rm odd integers}\}.\]
We find
\[ {\rm P.D.}\,{\rm ch}(P_\gamma) = \sum_{k \,\text{\rm odd}}\, W_k \;+\; \sum_{k \,\text{\rm odd}}\, k\,W_k\cup \frac{c_1}{2}\;\in H^1(X,\ZZ)\oplus H^3(X,\ZZ).\] 
Here all cocycles are integer cocycles.
The Poincar\'e dual of the $3$-cocycle 
$\sum k\,W_k\cup c_1/2$ is a $0$-cycle whose image under the map $H_0(X,\ZZ)\to H_0({\rm pt},\ZZ)=\ZZ$ is just the Fredholm index of $P_\gamma$.
This term in our $K$-cycle contains exactly the same information as the index formula of \cite{vE10c}.
The Poincar\'e dual of the term $\sum W_k$ is a $2$-cycle.
If we twist $P_\gamma$ by a complex vector bundle $F\to X$,
then this $2$-cycle will pair with $c_1(F)$, and the curvature of $F$ will contribute to the index of $F\otimes P_\gamma$.
The information contained in the term $\sum W_k$ cannot be derived or guessed from the formulas of \citelist{\cite{vE10c} \cite{vE10b}}.
The following formula highlights the gap between the result of the present paper and the earlier formula,
\[ \mathrm{Index}(F\otimes P_\gamma) = \mathrm{rank}(F)\cdot \mathrm{Index}\mathrm{P_\gamma} + \sum_{k \,\text{\rm odd}} \,\int_X \mathrm{ch}(\gamma-k)\wedge c_1(F).\]

\end{example}

\begin{example}\label{example}
The formula for the $K$-cycle of $P_\gamma$  easily extends to the case where $P_\gamma$ acts on sections in a trivial vector bundle.
With $X$ as above, let $r$ be a positive integer and let $\gamma$ be a $C^{\infty}$ map from $X$ to the $\CC$ vector space
of all $r\times r$ matrices, denoted  $M(r, \CC)$,   
\[
\gamma\; \colon \; X \to M(r, \CC)
\]
Then  --- entirely analogous to the case $r = 1$ ---  a differential operator $P_{\gamma}$ is given by :
\begin{align*}
P_\gamma &= \Delta_H \otimes I_r + iT \otimes \gamma \\
P_\gamma & \;\colon\; C^\infty(X,\CC^r)\to C^\infty(X,\CC^r)
\end{align*}
where (as usual) $I_r$ is the $r\times r$ identity matrix.   
$P_\gamma$ is elliptic in the Heisenberg calculus if and only if for all $x\in X$:
\[
\gamma(x)-\lambda I_r\quad \text{\rm is invertible for all} \;\lambda \in\{\dots, -n-4, -n-2, -n, n, n+2, n+4, \dots\}    
\]
As above, $P_\gamma$ determines an element $[P_\gamma]$ in $KK^0(C(X),\CC)$.
The $K$-cycle which solves the index problem for $[P_\gamma]$ is $(X^+\times S^1, E^+, \varphi)\sqcup(X^-\times S^1, E^-, \varphi)$
where $X^+\times S^1\sqcup X^-\times S^1$ and $\varphi$ are as above,
and 
\begin{align*}
E^+&= \bigoplus_{j=0}^N \left( \nu^\gamma_{2j+1}\otimes \varphi^*\mathrm{Sym}^j H^{1,0}\right) \\
E^-&= \bigoplus_{j=0}^N \left( \nu^\gamma_{-(2j+1)}\otimes \varphi^*\mathrm{Sym}^j H^{0,1} \right)
\end{align*}
Here $\nu^\gamma_{2j+1}$ (resp. $\nu^\gamma_{-(2j+1)}$) is the $\CC$ vector bundle of fiber dimension $k$ on $X\times S^1$
obtained by doing a clutching construction using $\gamma - (2j+1)I_k$ (resp. $\gamma +(2j+1)I_k$).
\end{example}


\bibliographystyle{amsplain}

\bibliography{MyBibfile}

\newpage

\end{document}